\documentclass{elsarticle}
\usepackage{amssymb}
\usepackage{amsmath}
\usepackage{xcolor}
\usepackage{amsthm}
\usepackage{latexsym}
\newtheorem{theorem}{Theorem}[section]

\newtheorem{corollary}[theorem]{Corollary}
\newtheorem{proposition}[theorem]{Proposition}

\theoremstyle{definition}
\newtheorem{definition}[theorem]{Definition}
\newtheorem{example}[theorem]{Example}
\newtheorem{remark}[theorem]{\bf Remark}

\newcommand{\oomit}[1]{{}}

\newcommand{\bN}{\ensuremath{\mathbb{N}}}

\newcommand{\bC}{\ensuremath{\mathbb{C}}}

\newcommand{\cF}{\ensuremath{\mathcal{F}}}

\newcommand{\lbox}{\lower 1pt \hbox{$\,\ssq\,$} }

\renewcommand{\phi}{\ensuremath{\varphi}}
\renewcommand{\epsilon}{\ensuremath{\varepsilon}}

\renewcommand{\leq}{\ensuremath{\leqslant}}
\renewcommand{\geq}{\ensuremath{\geqslant}}

\newcommand{\ssq}{\raise 1pt\hbox{$\scriptscriptstyle \square$}}

\newcommand \dstar{\mbox{\rm (\kern-0pt $\genfrac{}{}{0pt}{}{\raisebox{0pt}{$\star$}}{\raisebox{1pt}{$\star$}}$\rm)}}

\newcounter{smallromans}

\newcommand{\C}{\mathcal{C}}

\newcommand{\tensor}{\widehat{\otimes}}

\newcommand{\BAI}{bounded approx\-imate identity}

\newcommand{\AI}{approx\-imate identity}

 \newcommand{\BAAy}{bounded approx\-imate amen\-ability}
\newcommand{\ApA}{approx\-imately amen\-able}
\newcommand{\ApsA}{approx\-imately semi-\-amen\-able}
\newcommand{\ApsCy}{approx\-imate semi-\-contract\-ibility}
\newcommand{\ApsAy}{approx\-imate semi-\-amen\-ability}
\newcommand{\ApsC}{approx\-imately semi-\-contract\-ible}

\newcommand{\ApAy}{approx\-imate amen\-ability}
 \newcommand{\BApA}{boundedly approx\-imately amen\-able}
 \newcommand{\BApsA}{boundedly approx\-imately semi-amen\-able}
  \newcommand{\BApSAy}{bounded approx\-imate semi-amen\-ability}
\newcommand{\BASC}{boundedly approx\-imately semi-contract\-ible}
  \newcommand{\BApSCy}{bounded approx\-imate semi-contract\-ibility}
\newcommand{\BAC}{boundedly approx\-imately contract\-ible}
 \newcommand{\AC}{approx\-imately contract\-ible}
 \newcommand{\ACy}{approx\-imate contract\-ibility}
  \newcommand{\BACy}{bounded approx\-imate contract\-ibility}

\begin{document}




\author{F. Ghahramani\fnref{labelG}}
\address{Department of Mathematics, 
University of Manitoba, Winnipeg, R3T 2N2, Canada}
\ead{Fereidoun.Ghahramani@umanitoba.ca}

\author{R. J. Loy\corref{cor2}}
\address{Mathematical Sciences Institute, 
Hanna Neumann Building 145,  Science Road, Australian National University, 
Canberra, ACT 2601, Australia}
\ead{Rick.Loy@anu.edu.au}

\title
{Approximate semi-amenability of Banach  algebras} 

\fntext[labelG]{Supported by NSERC grant 05476-2017}



\newcommand\timestring{\begingroup
\count0=\time \divide\count0 by 60
\count2 = \count0
\count4 = \time \multiply\count0 by 60
\advance\count4 by -\count0
\ifnum\count4 <10 \toks1={0}%
\else \toks1 = {}%
\fi
\ifnum\count2<12 \toks0 = {a.m.}%
\else \toks0 = {p.m.}%
\advance\count2 by -12
\fi
\ifnum\count2=0 \count2 = 12
\fi
\number\count2:\the\toks1 \number\count4
\thinspace \the\toks0
\endgroup}

\newcommand\timenow{ \timestring \enskip
{\ifcase\month\or January\or February\or March\or 
  April\or May\or June\or July\or August\or Sepember\or
  October\or November\or December\fi}
\number\day}

\begin{abstract} 

In recent work  of the authors the notion of a derivation being approximately semi-inner arose as a tool for investigating (approximate) amenability questions  for  Banach algebras.  Here we investigate this property in its own right, together with the consequent one of approximately semi-amenability. Under certain hypotheses regarding approximate identities this new notion is the same as approximate amenability,  but more generally it covers some important classes of algebras which are not approximately amenable, in particular Segal algebras on amenable SIN-groups. 
 \end{abstract}


\maketitle 

\section{Introduction}  \label{Introduction}
The concept of amenability for a Banach algebra, introduced by Johnson in 
\cite{John}, has proved to be of enormous importance in Banach algebra theory. 
In \cite{GhaLoy}, and subsequently in \cite{GhaLoyZh},  several modifications of this notion were introduced, in particular that of approximate amenability;  and  much work has been done in the last 10 years or so, \cite{DLZ, CGZ, CG, GhaSt, DL, GhaR1, GhaR2}, for example. See also \cite{Zhang} for a recent survey. More recently, \cite{GhaLoy2} investigated the situation for tensor products, and,\emph{en passant}, introduced the notions to be considered here.

\vskip 2mm
Let $A$ be an algebra, $X$ an $A$-bimodule.  A \emph{derivation} is a linear map $D: A \to X$ such that
\[
D(ab)  = a\,\cdot \, D(b) + D(a)\,\cdot \,b \quad (a,b \in A)\,.
\]
For $\xi \in X$, set ${\rm ad}_{\xi}: a \mapsto  a\,\cdot \, \xi - \xi\,\cdot \,a,\,\; A \to X$. Then ${\rm ad}_{\xi}$ is a derivation; these are  the \emph{inner} derivations.

Let $A$ be a Banach algebra, $X$ be a
Banach $A$-bimodule.  A continuous derivation $D:A\to X$ is \emph{approximately inner} if there is
a net $(\xi_i)$ in $X$ such that
\[
D(a) =
\lim_{i}(a\,\cdot\,\xi_i  - \xi_i\,\cdot\, a)\quad (a \in A)\,,
\]
so that $D =\lim_{i}{\rm ad}_{x_i}$ in the strong-operator topology  of ${\mathcal B}(A , X)$.

\begin{definition}\cite{GhaLoy, GhaLoyZh} \label{app} Let $A$ be a Banach algebra. Then
$A$ is \emph{approximately amenable} (resp. \emph{approximately contractible})  if, for
each Banach $A$-bimodule $X$, every continuous derivation $D:A\to X^{*}$ (resp. $D:A\to X$), is approximately inner.  If it is always possible to choose the approximating net $({\rm ad}_{\xi_i})$ to be bounded (with the bound dependent only on $D$) then $A$ is \emph{boundedly} \ApA\ (resp. \emph{boundedly} \AC).
\end{definition}

Of course $A$ is \emph{amenable} (resp. \emph{contractible}) if every continuous derivation $D:A\to X^{*}$ (resp. $D:A\to X$), is inner, for every Banach $A$-bimodule $X$.

\vskip1mm 


 During final preparations of this paper, the article \cite{KaA} appeared on the arXiv.  The authors there refer to \cite{GhaLoy2} for the introduction of the terminology semi-inner and semi-amenability.  However they have overlooked \cite[Corollary 3.5 and Remark 3.6]{GhaLoy2}, \emph{cf.} Proposition \ref{equality1} (iii) below, which shows that the notions of semi-amenability and amenability coincide. Similarly, semi-contractibility and contractibility coincide.  In particular, this answers the final Problem in \cite{KaA}.  The paper \cite{GADM}, by the same authors and one other, has the same oversight.  More recently,  the same authors have listed \cite{KAM} on the arXiv.  Unfortunately, there are numerous errors in  \cite{KAM}.


\section{Semi-inner derivations}\label{new method}
The following definition was given in \cite{GhaLoy2}. This type of mapping is not new, see \cite{AM, Br, CM}, where they are called `generalized inner (derivations)', but the additional condition that they actually be derivations puts a different perspective on the situation.
\begin{definition}
Let $A$ be an algebra, $X$ an $A$-bimodule.  A linear map $D:A\to X$ is \emph{semi-inner} if there are $m, n\in X$ such that 
\begin{equation}\label{addefn}
D(a) = {\rm ad}_{\xi, \eta}(a): = a\cdot \xi- \eta\cdot a \qquad (a\in A)\,.
\end{equation}
\end{definition}

Supposing that $D$ is in fact a derivation, then $\xi - \eta$  is highly constrained.  The derivation identity  shows that  if ${\rm ad}_{\xi, \eta}$ is a  derivation then
\begin{equation}\label{annihilation}
 a\cdot (\xi - \eta)\cdot b = 0\qquad (a, b\in A)\,.
\end{equation}
Conversely, if $\xi, \eta \in X$ satisfy \eqref{annihilation} then it is immediate that ${\rm ad}_{\xi, \eta}$ is a derivation.

It was shown in \cite{GhaLoy2} that semi-inner is indeed a strictly weaker notion than inner, but that  
 in many situations of interest the notions coincide.  
\begin{proposition} {\rm (\cite[Proposition 3.3, Corollary 3.5]{GhaLoy2})}  \label{equality1}  
Let $A$ be a Banach algebra, $X$ a Banach $A$-bimodule, $D:A\to X$ (resp. $D:A\to X^{*}$) a derivation.
 Suppose that $D$ is semi-inner. Then in each of the following cases  $D$ is inner:
		\begin{enumerate}
			\item[(i)]  $A$ has left and right approximate identities for $X$;
			\item[(ii)]  $D:A\to X^{*}$ and $X$ is neo-unital; 
			\item[(iii)] $D:A\to X^{*}$ and $A$ has a \BAI\,.
		\end{enumerate}
		In particular, if $A$ is semi-amenable then $A$ is amenable.
 \end{proposition}\hfill $\Box$
  
\begin{definition}
For a Banach algebra $A$ and a  Banach $A$-bimodule $X$, a continuous derivation  $D:A\to X$ is \emph{approximately semi-inner} if there are nets $(\xi_i), (\eta_i)$ in $X$ with
\begin{equation}\label{asi}
D(a) = \lim_i(a\cdot \xi_i - \eta_i\cdot a) \qquad (a\in A)\,,
\end{equation}
that is, $D$ is the limit in the strong operator topology of the net $({\rm ad}_{\xi_{i}, \eta_{i}} )$.
\end{definition}

 Note that in distinction to when a derivation is approximately inner, the approximating operators here are not \emph{a priori} supposed to be derivations.  However, if
\begin{equation}\label{appform}
D(a)= \lim_i (a\cdot \xi'_i - \eta'_i\cdot a)\qquad (a\in A)\,,
 \end{equation}
and $D$ is a  derivation, then the product rule yields
a parallel to \eqref{annihilation}:
\begin{equation}\label{limits}
    \lim_i(a\cdot(\xi'_i - \eta'_i)\cdot b) = 0\qquad (a, b\in A)\,.    
\end{equation}
 In considering approximately semi-inner derivations, both \eqref{appform} and \eqref{limits}  need to be taken into account.
\oomit{ \begin{proposition}\label{identity}
 Let $A$ be a Banach algebra with identity. Then any approximately semi-inner derivation from $A$ into a Banach $A$-bimodule is approximately inner.
 \end{proposition}
 \begin{proof}
 Suppose  that $A$ has an identity $e$ and  $D:A\to X$ is a derivation into a Banach $A$-bimodule $X$. The usual decomposition, \cite[Lemma 2.3]{GhaLoy},  shows that after adjusting by inner derivations, we may suppose that $D$ maps into $e\cdot X\cdot e$.  Supposing that $D$ is approximately semi-inner, there are nets $(\xi_{i}), (\eta_{i})$ in $e\cdot X\cdot e$ such that for $a, b\in A$,  $D(a)=\lim{i}\xi_{i}(a\cdot \xi_{i} - \eta_{i}\cdot a)$ and $\lim{i}\xi_{i} a\cdot(\xi_{i} - \eta_{i})\cdot b= 0$.   Putting $a =b = e$ we have $e\cdot (\xi_{i} - \eta_{i})\cdot e\to 0$, that is, $\xi_{i} -\eta_{i}\to 0$.  Thus $D= \lim{i}\xi_{i}{\rm ad}_{\xi_{i}}$ is approximately inner. 
 \end{proof} 
}

\begin{proposition}\label{cbaistrong}
 Let $A$ be a Banach algebra with  a central \BAI.  Then any approximately semi-inner derivation from $A$ into a Banach $A$-bimodule is approximately inner.  In particular, the result holds in the case when $A$ is unital.

\end{proposition}

\begin{proof}

Let $X$ be a general Banach $A$-bimodule, $D:A\to X$ a continuous derivation. For the moment consider $D:A\to X^{**}$.  Let $(e_\alpha)$ be a central  \BAI\ for $A$.  By Cohen's factorization theorem, $X^{*}_{ess} = A\cdot X^{*}\cdot A$ is a neo-unital $A$-bimodule.  Let $E$ be a limit point in the weak$^*$-operator topology of the left multiplication operators on $X^{**}$ by the elements $e_{\alpha}$, $F$ similarly for right multiplication.  Then $E$ and $F$ are commuting $A$-bimodule morphisms of  $X^{**}$ into itself, and give  a decomposition
\begin{equation}\label{decomp411}
X^{**} = EFX^{**} \oplus E(I-F)X^{**} \oplus (I-E)X^{**}\,.
\end{equation}
Correspondingly, set
$$
D_1 = EFD, D_2 = E(I-F)D, D_3 = (I-E)D\,.
$$
These are easily seen to be derivations into the corresponding summands in \eqref{decomp411}.  Since $A$ has trivial action on the right of $E(I-F)X^{**}$, and has a \BAI, we conclude that $D_{2}$ is approximately inner, with a bounded net of implementing elements (whence $D_{2}$ is inner).  Similarly for $D_{3}$.

Thus we may reduce to $D_1$ mapping into the bimodule $EFX^{**}$ which can be viewed as the dual module of the 
 module $X^{*}_{ess}$. 

Now suppose that $D$ is approximately semi-inner. We have nets $(\xi_{j}), (\eta_{j})$ in $X$ such that for any $a, b\in A$, 
\begin{equation}\label{nets1}
D(a) = \lim_{j}(a\cdot  \xi_{j} - \eta_{j}\cdot a), \qquad a\cdot (\xi_{j} - \eta_{j})\cdot b\to 0\,.
\end{equation}

Then for $a\in A$,
\begin{eqnarray}
D_{1}(a) &=& (w^{*}-\lim_{\alpha})(w^{*} -\lim_{\beta})e_{\alpha}D(a)e_{\beta}\nonumber\\
&=&
(w^{*}-\lim_{\alpha})(w^{*} -\lim_{\beta})\lim_{i}\Big(e_{\alpha}(a\cdot \xi_{i} - \eta_{i}\cdot a)e_{\beta}\Big)\,. \label{Dlimit1}
\end{eqnarray}
 
 Then \eqref{nets1} and \eqref{Dlimit1} give, using centrality of the \BAI, 
 $$
 D_{1}(a)  =(w^{*}-\lim_{\alpha})(w^{*} -\lim_{\beta})\lim_{i}\Big(a\cdot(e_{\alpha}\cdot \xi_{i}\cdot e_{\beta}) - (e_{\alpha}\cdot \xi_{i}\cdot e_{\beta})\cdot a\Big)\,. $$
 By considering finite subsets of $A$ and $X^{*}$ we can find a net $(x^{**}_{\delta}) \subset X^{**}$ such that 
 $$ D_{1}(a) = (w^{*}-\lim_{\delta})  (a\cdot x^{**}_{\delta} - x^{**}_{\delta}\cdot a) \ \ \ (a\in A). $$
 Since $D_{2}$ and $D_{3}$ are approximately inner we deduce that $D$ is weak$^{*}$-approximately inner.
 Again the standard method of considering finite subsets of $A$ and $X^{*}$ together with Goldstine theorem and a version of Hahn-Banach theorem, give a net $(x_{\gamma})\subset X$ such that
 $$
 D(a) = \lim_{\gamma} (a\cdot x_{\gamma} - x_{\gamma}\cdot a)\,, \qquad (a\in A)\,
 $$ and the proof is complete. \end{proof}
   
 We make the following definitions.  Each notion is clearly stronger than the one following it; in fact all three will turn out to be equivalent.
\begin{definition}\label{variants}
The Banach algebra $A$ is 
\begin{itemize}
 \item \emph{approximately semi-contractible} if for any Banach $A$-bimodule $X$, any continuous derivation $D :A\to X$ is approximately semi-inner.
\item \emph{approximately semi-amenable} if for any Banach $A$-bimodule $X$, any continuous derivation $D :A\to X^{*}$ is approximately semi-inner.
 \item \emph{weak$^{*}$-approximately semi-amenable} if for any Banach $A$-bimodule $X$, any continuous derivation $D :A\to X^{*}$ can be approximated weak$^{*}$ by semi-inner mappings.
\end{itemize}
\end{definition}
 
 \begin{theorem}\label{semicont}
The Banach algebra $A$ is approximately semi-contractible if and only if there are nets  $(\xi_i)$ and $(\eta_i)$ in $A^{\#}\tensor  A^{\#}$ such that, with $e$  the adjoined identity,
\begin{enumerate}
\item $a\cdot\xi_i-\eta_{j} \cdot a \to 0\,,\quad (a\in A)$\,,
\item $ \pi(\xi_i) \to e$ and $ \pi(\eta_i) \to e$\,,
\end{enumerate}
\end{theorem}

\begin{proof}  Suppose that $A$ is approximately semi-contractible. Consider the derivation $D:A\to A^{\#}\tensor  A^{\#}$ given by $D(a) = a\otimes e - e\otimes a$.  $D$ maps into ${\rm ker}\,\pi$, so there exist $(\xi_i')$, $(n_i')$ in ${\rm ker}\,\pi$, such that $D = \lim_i {\rm ad}_{\xi_i', n_i'}$ in the strong operator topology. Set $\xi_i= e\otimes e -\xi_i'$, $\eta_i = e\otimes e -\eta_i'$. Then for $a\in A$,
$$
a\cdot \xi_i - \eta_i\cdot a =a\otimes e - a\cdot \xi_i' -e\otimes a +n_i'\cdot a \to 0\,,
$$
and $\pi(\xi_i) = \pi(\eta_i) = e$.    (From \eqref{limits}, $a\cdot (\xi_{i}' - \eta_{i}')\cdot b \to 0$ for $a, b\in A$, so  we also have  $ a\cdot(\xi_i - \eta_i)\cdot b \to 0$ for $a, b\in A$.)


Conversely, suppose $(\xi_i)$ and $(\eta_i)$ satisfy clauses 1 and 2. Let $D:A\to X$ be a continuous derivation into a Banach $A$-bimodule $X$. Setting $ e\cdot x = x\cdot e = x$ for $x\in X$ makes $X$ into a Banach $A^{\#}$-bimodule.  Extend $D$ to $A^{\#}$ by setting $D(e) = 0$.  Define $\varphi:A^{\#}\tensor A^{\#}\to X$ by $\varphi(a\otimes b) = aD(b)$.

\

Then for $
\xi = \sum_k a_k\otimes b_k$, with $\pi(\xi) = e$,
\begin{eqnarray*}
&&\varphi(\xi\cdot a) = \varphi(\sum a_k\otimes b_ka) =\sum a_k D(b_k a) \cdot a + \pi(\xi) D(a)= \varphi(\xi)\cdot a +D(a)\\
 &&\kern -2em \textrm{and}\\ 
&&\varphi(a\cdot \xi) =a\cdot \xi\,.
\end{eqnarray*}
Thus
$$
\varphi(a\cdot \xi_i) -\varphi(\eta_i \cdot a) = a\cdot\varphi(\xi_i)- \varphi(\eta_i)\cdot a - \pi(\xi_{i})\cdot D(a) \to 0\,.
$$
But LHS $\to 0$ by hypothesis, whence
$$
D(a) = \lim_i \big(a\cdot \varphi(\xi_i) - \varphi(\eta_i)\cdot a\big)\,.
$$
\end{proof}

\begin{remark}\label{eqapprox}  The above argument is standard, the details have been included as they show that we may replace clause 2 above by the equality  $ \pi(\xi_i) = \pi(\eta_i)= e$\,.  
Analogously, the same holds in the next two results, so these are stated with the stronger clause 2.
\end{remark}

Thanks to Theorem \ref{equiv} below, Theorem \ref{semicont} and the next two results are giving different formulations of the same concept.  The proofs are minor variants of those of Theorem \ref{semicont}, making use of the natural projection $X^{***}\to X^{*}$
for any Banach space $X$.

 \begin{theorem}\label{semiamen}
The Banach algebra $A$ is approximately semi-amenable if and only if there are nets  $(\xi_i)$ and $(\eta_i)$ in $(A^{\#}\tensor  A^{\#})^{**}$ such that, with $e$  the adjoined identity, and $\pi: A^{\#}\tensor A^{\#}\to A^{\#}$ the product map,
\begin{enumerate}
\item  $a\cdot \xi_i - \eta_i \cdot a \to 0\,,\quad (a\in A)$\,,
\item $ \pi^{**}(\xi_i) = \pi^{**}(\eta_i) = e\,.$
\end{enumerate}  \vskip -\baselineskip\vskip -2mm  \qed

\end{theorem}

\begin{remark}
This result immediately shows that in the unital  case \ApsAy\ implies \ApAy, a weak form of Theorem \ref{semi+baa=approx} below.

\end{remark}
\oomit{
\begin{proof}  Suppose that $A$ is approximately semi-amenable. The same argument as in Theorem \ref{semicont} works here.

Conversely, arguing as before for a derivation $D:A\to X^{*}$, we get, for ${m\in A^{\#}\tensor A^{\#}}$, $a\in A$,
$$
\varphi(m\cdot a)  = \varphi(m)\cdot a +\pi(m)D(a)\,, \quad  \varphi(a\cdot m) =a\cdot m\,.
$$
It follows that for $M\in (A^{\#}\tensor A^{\#})^{**}$ with $\pi^{**}(M) = e$, $a\in A$, and with ${j:X^{*}\to X^{***}}$ the natural embedding,
$$
\varphi^{**}(M\cdot a)  = \varphi^{**}(M)\cdot a +jD(a)\,, \quad  \varphi^{**}(a\cdot M) =a\cdot M\,.
$$
Thus
\marginpar{\bf \color{red}{proof could go}}
$$
\varphi^{**}(a\cdot \xi_i) -\varphi^{**}(N_i \cdot a) = a\cdot\varphi^{**}(\xi_i)- \varphi(N_i)\cdot a - jD(a) \to 0\,.
$$
But LHS $\to 0$ by hypothesis, whence
$$
jD(a) = \lim{i}\xi_i \big(a\cdot \varphi^{**}(\xi_i) - \varphi^{**}(N_i)\cdot a\big)\,.
$$

Applying the natural projection $Q:X^{***}\to X^{*}$, we obtain
$$
D(a) = \lim{i}\xi_i \big(a\cdot Q(\varphi^{**}(\xi_i)) - Q(\varphi^{**}(N_i))\cdot a\big)\,.
$$

\end{proof}
}

 \begin{theorem}\label{semiweak}
The Banach algebra $A$ is weak$^{*}$-approximately semi-amenable if and only if there are nets  $(\xi_i)$ and $(\eta_i)$ in $(A^{\#}\tensor  A^{\#})^{**}$ such that, with $e$  the adjoined identity,
\begin{enumerate}
\item  $a\cdot \xi_i - \eta_i \cdot a \to 0\ {\rm weak}^{*}\,,\quad (a\in A)$\,,
\item $ \pi^{**}(\xi_i) = \pi^{**}(\eta_i) = e\,.$
\end{enumerate} \vskip -\baselineskip\vskip -2mm  \qed
\end{theorem}
\oomit{
\begin{proof}  Suppose that $A$ is weak$^{*}$ approximately semi-amenable. The same argument as in Theorem \ref{semicont} works again.

Conversely, arguing as before for a derivation $D:A\to X^{*}$ we get, with $\varphi$ and $Q$ as in Theorem \ref{semiamen}, in the weak$^{*}$-topology in $X^{***}$,
$$
jD(a) = \text{weak}^{*}-\lim{i}\xi_i \big(a\cdot \varphi^{**}(\xi_i) - \varphi^{**}(N_i)\cdot a\big)\,,
$$
whence
\marginpar{\bf \color{red}{proof could go}}
$$
D(a) = {\rm weak}-\lim{i}\xi_i \big(a\cdot Q(\varphi^{**}(\xi_i)) - Q(\varphi^{**}(N_i))\cdot a\big)\,.
$$
And so certainly,
$$
D(a) = \text{weak}^{*}-\lim{i}\xi_i \big(a\cdot Q(\varphi^{**}(\xi_i)) - Q(\varphi^{**}(N_i))\cdot a\big)\,.
$$
\end{proof}
}


 
 We now have  the following parallel  to \cite[Theorem 2.1]{GhaLoyZh}. It shows that the variants of Definition \ref{variants} are in fact the same, and its proof uses an argument somewhat similar to the argument  of Proposition \ref {cbaistrong}.
 
\begin{theorem}\label{equiv} 
For a Banach algebra $A$, the following are equivalent:
\begin{enumerate}
\item[(i)] $A$ is approximately semi-contractible,
\item[(ii)]  $A$ is approximately semi-amenable,
\item[(iii)]  $A$ is weak$^*$ approximately semi-amenable.
\end{enumerate}  \vskip -\baselineskip\vskip -2mm  \qed
\end{theorem}

Henceforth we will use the terms \ApsA\ and  \ApsC\ interchangeably. 

\begin{theorem}\label{semi+baa=approx}

In the presence of a central bounded approximate identity, approximate semi-amenability and approximate amenability are the same.\end{theorem}

\begin{proof}This a direct consequence of  Proposition \ref{cbaistrong}. \end{proof}    

 \vskip 2mm
Finally in this section, a couple of results which parallel the \ApA\ case. 

\begin{theorem} \label{parallel}
Suppose that the Banach algebra $A$ is approximately semi-amenable.  Then $A$ has left and right approximate identities.  If $J$ is a closed (two-sided) ideal in $A$ with a bounded approximate identity,  then $J$ is \ApsA\,.
\end{theorem}
\begin{proof}
The first result uses standard argument with $A$ as a Banach $A$-bimodule  with trivial action on one side.  For the second, just use the standard  argument of reducing to the neo-unital case,  and extending to the multiplier algebra.  See \cite{John, GhaLoy}.
\end{proof}

It is the same argument that shows that semi-amenable implies the existence of a  \BAI.
The following strengthens part of Theorem \ref{parallel}, we do not know whether $J$ is always \ApsA.  See Theorem \ref{cai} below for a sufficient condition for this to be the case.
\begin{theorem}\label{proj}
Suppose that $A$ is \ApsA, and that $J$ is a closed complemented ideal of $A$.   Then $J$ has left and right approximate identities.
\end{theorem}
\begin{proof}  Let $P:A\to J$ be a bounded projection. Adjoin an identity $e$ to $A$, and extend $P:A^{\#} \to J$ by setting $P(e) = 0$.   Define $\Phi: A^{\#}\tensor A^{\#} \to J$ by $\Phi(a\otimes b) = aP(b)$.   Clearly $\|\Phi\|\leq \|P\|$. Take a finite set $F\subset J$ and set $K = \max\{1, \|f\|: f\in F\}$, and take  $\varepsilon > 0$.  Then by \ApsAy of $A$ there is $\xi, \eta\in A^{\#}\tensor A^{\#}$ such that $\|f\cdot \xi - \eta\cdot f\|<\varepsilon/(2\|P\|)$, $f\in F$,  $\|\pi(\xi)-e\| <\varepsilon/(2K)$,  and $\|\pi(\eta)-e\| <\varepsilon/(2K)$.  We may assume that for some $k$, $\xi=\sum_{j=1}^{k} a_{j} \otimes b_{j}$ and  $\eta=\sum_{j=1}^{k} a_{j}' \otimes b_{j}'$ for some $a_{j}, b_{j}, a_{j}', b_{j}' \in A^{\#}$.  Then for $f\in F$,
  \begin{eqnarray*}
 &&\kern -1cm  \|f \sum_{j=1}^{k} a_{j}P(b_{j}) - f\| = \|\Phi(f\cdot \xi) - f\|\\
  & & \kern 1cm \leq\|\Phi(f\cdot \xi-\eta\cdot f)\| + \left\|\sum_{j=1}^{k}a_{j}'[P(b_{j}'f) - b_{j}'f]\right\|+ \|\pi(\eta)f - f\|\\
   & & \kern 1cm \leq \varepsilon\,.
  \end{eqnarray*}
  
Thus $J$ has a right approximate identity. Interchanging the roles of $\xi$ and $\eta$ gives a left approximate identity. 
 \end{proof}
 
 \begin{remark} Recall that  a subspace $E$ of a normed space $X$ is approximately complemented in $X$ is for any compact subset $K$ of $E$ and any $\varepsilon > 0$ there is a continuous operator $P:X\to E$ such that $\|x-Px\| < \varepsilon$ for all $x\in K$, \cite{Zhang1}.  It is boundedly approximately complemented if this can be done with $\|P\|\leq K$ for some fixed $K > 0$.  In the bounded case it is easily seen that the condition holds if and only if it holds for all finite sets $E\subset X$.  A small variant to the proof shows that Theorem \ref{proj} holds with $J$ boundedly approximately complemented.
 
 \end{remark}

\section{Bounded variants}\label{bounded}  

\begin{definition}
For a Banach algebra $A$ and a  Banach $A$-bimodule $X$, a continuous derivation  $D:A\to X$ is \emph{boundedly approximately semi-inner} if there is a constant $K > 0$ and nets $(\xi_i), (\eta_i)$ in $X$ such that
\begin{enumerate}
\item[(i)] $D(a) = \lim_i(a\cdot \xi_i - \eta_i\cdot a)\, (a\in A)$\,; and
\item[(ii)] $\|{\rm ad}_{\xi_{i},\eta_{i}}\|\leq K$, for all $i$\,.
\end{enumerate}
\end{definition}

Parallel to \eqref{limits} we have as a consequence
\begin{enumerate}
\item[(iii)] $a\cdot (\xi_{i} - \eta_{i})\cdot b\to 0$, and for all $i$, $\|a\cdot (\xi_{i} - \eta_{i})\cdot b\|\leq 2K\|ab\|$,  ($a, b\in A)$\,.
\end{enumerate}

We will see below in Example \ref{ellpex} that such derivations need not be approximately inner.

We have the following variant of Definition \ref{variants}. Once again, each notion is clearly stronger than the one following;  however this time there is no equivalence.
\begin{definition}\label{bddvariants}
The Banach algebra $A$ is 
\begin{itemize}
 \item \emph{boundedly approximately semi-contractible} if for any Banach $A$-bimodule $X$, any continuous derivation $D :A\to X$ is boundedly approximately semi-inner.
\item \emph{boundedly approximately semi-amenable} if for any Banach $A$-bimodule $X$, any continuous derivation $D :A\to X^{*}$ is boundedly approximately semi-inner.
 \item \emph{weak$^{*}$-boundedly approximately semi-amenable} if for any Banach $A$-bimodule $X$, any continuous derivation $D :A\to X^{*}$ can be boundedly approximated weak$^{*}$ by semi-inner mappings.
\end{itemize}
\end{definition} 

There are corresponding variants of Theorems \ref{semicont}, \ref{semiamen} and \ref{semiweak}, just by adding in condition (ii).  We will only state the first two of these.

 \begin{theorem}\label{bddsemicont}
The Banach algebra $A$ is boundedly approximately semi-contract\-ible if and only if there is a constant $K$, and nets  $(\xi_i)$ and $(\eta_i)$ in $A^{\#}\tensor  A^{\#}$ such that
\begin{enumerate}
\item  $a\cdot \xi_i - \eta_i \cdot a \to 0\,,\quad (a\in A)$\,,
\item for all $i$, $\|{\rm ad}_{\xi_{i},\eta_{i}}\|\leq K$\,,
\item for all $i$, $\pi(\xi_i) = e$ and $ \pi(\eta_i) = e$\,.
\end{enumerate}
where $e$ is the adjoined identity. \hfill $\Box$
\end{theorem}

\begin{definition}  The {\emph{\BApSCy\ constant of a Banach algebra $A$}}, denoted $K_{BASC}(A)$, is the infimum of all the $K$ satisfying the conditions of Theorem \ref{bddsemicont}, or $+\infty$ if there is no such $K$.
 \end{definition}

\emph{A priori,} $K_{BASC}(A)$ could be $0$.  Supposing it were $0$, then for each $k$, there are nets $(\xi^{(k)}_{i(k)})$ and $(\eta^{(k)}_{i(k)})$ such that for each $a\in A$,
\[
\|a\cdot \xi^{(k)}_{i(k)} - \eta^{(k)}_{i(k)} \cdot a \| \leq \|a\|/k\,.
\] 
But then choosing any element $j_{k}$ from the $i(k)$ from the $k$-th index set, and setting $\xi'_{k} = \xi^{(k)}_{j_{k}}$ and $\eta'_{k} = \eta^{(k)}_{j_{k}}$, we have
\[
\|a\cdot \xi'_{k} - \eta'_{k} \cdot a \| \leq \|a\|/k\,, \quad (a\in A)\,.
\] 
Thus for these sequences we have clause 1 holding uniformly on $\|a\| = 1$.  So by Theorem \ref{unifscont} below, $A$ is contractible.

 \begin{theorem}\label{bddsemiamen}
The Banach algebra $A$ is boundedly approximately semi-amenable if and only if there are nets  $(\xi_i)$ and $(\eta_i)$ in $(A^{\#}\tensor  A^{\#})^{**}$ such that, with  $e$ the adjoined identity.
\begin{enumerate}
\item  $a\cdot \xi_i - \eta_i \cdot a \to 0\,,\quad (a\in A)$\,,
\item for all $i$, $\|{\rm ad}_{\xi_{i},\eta_{i}}\|\leq K$\,,
\item  for all $i$, $\pi^{**}(\xi_i) = e$ and $ \pi^{**}(\eta_i) = e$\,.\hfill $\Box$

\end{enumerate} 
\end{theorem}

\begin{definition}  The { \emph{\BApSAy\ constant of a Banach algebra $A$}}, denoted $K_{BASC}(A)$, is the infimum of all the $K$ satisfying the conditions of Theorem \ref{bddsemiamen}, or $+\infty$ if there is no such $K$.
 \end{definition}


Once again, if  $K_{BASA}(A)= 0$ then $A$ is amenable (and conversely).

\oomit{  \begin{theorem}\label{bddsemiweak}
$A$ is weak$^{*}$-boundedly approximately semi-amenable if and only if there are nets  $(\xi_i)$ and $(N_i)$ in $(A^{\#}\tensor  A^{\#})^{**}$ such that
\begin{enumerate}
\item  $a\cdot \xi_i - N_i \cdot a \to 0\ {\rm weak}^{*}\,,\quad (a\in A)$\,,
\item for all $i$, $\|{\rm ad}_{\xi_{i},\eta_{i}}\|\leq K$\,,
\item $ \pi^{**}(\xi_i) = \pi^{**}(N_i) = e\,.$
\end{enumerate}
where $e$ is the adjoined identity.
\end{theorem}
}

 In \cite{GhaLoyZh} it is shown that \ApAy\ and \ACy\ are the same, and Theorem \ref{equiv} above shows the result holds with the qualifier  `semi'.  On the other hand, \cite{GhaR2} shows that \ApA\ and \BApA\ are distinct notions. 

 \begin{proposition} \label{difference}  Bounded approximate semi-amenability is weaker than \BApSCy. 
 \end{proposition}

\begin{proof}
Note firstly that any  boundedly approximately semi-contractible algebra has multiplier bounded left and right approximate identities, this follows from a variant of Theorem \ref{parallel}. 
So if the algebra was \BApA, then by \cite[Theorem 3.3]{CGZ} it would have a (two-sided) \BAI.  Thus the example of \cite{GhaR1} is a Banach algebra $A$ which is \BApA\ but not boundedly approximately semi-contractible. Being \BApA,  it is trivially \BApsA, so this latter notion is therefore strictly weaker than \BApSCy.
\end{proof}

In \cite[Example 3.7]{GhaLoy2} we showed that the algebras $\ell^{p}$, $(1\leq p < \infty)$, with pointwise product, are \ApsA.  In fact we can say more:
\begin{example} \label{ellpex}
 The algebras $\ell^{p}$, $1\leq p < \infty$ are \BASC, with $K_{BAC}(\ell^{p}) = 2$.  For $k\in \bN$ define 
\[
e_{k}(j) = 
\begin{cases}
1 & 1\leq j\leq k\\
0 & {\rm otherwise},
\end{cases}
\]
 and define $\xi_{k}, \eta_{k}\in (\ell^{p})^{\#}\otimes  (\ell^{p})^{\#}$ by
 \begin{eqnarray*}
 \xi_{k} = e\otimes e - e_{k}\otimes e + \sum_{j=1}^{k}\delta_{j}\otimes \delta_{j}\,,\quad
\eta_{k} = e\otimes e - e\otimes e_{k} + \sum_{j=1}^{k}\delta_{j}\otimes \delta_{j}\,,\\
\end{eqnarray*}
where $e$ is the adjoined identity.  Then $\pi(\xi_{k}) = \pi(\eta_{k}) = e$, and for each $a= (a_{j)}\in \ell^{p}$, 
\begin{eqnarray*}
a\cdot \xi_{k} &=&  a\otimes e - ae_{k}\otimes e + \sum_{j=1}^{k} a_{j} \delta_{j} \otimes \delta_{j}\,,\\
\eta_{k}\cdot a &= & e\otimes a - e\otimes e_{k}a + \sum_{j=1}^{k} a_{j} \delta_{j} \otimes \delta_{j}\,,\\ 
\end{eqnarray*}
whence
 \[
 a\cdot \xi_{k} - \eta_{k}\cdot a = (a-e_{k}a)\otimes e - e \otimes (a - ae_{k}) \to 0, \quad \textrm{as}\ k\to \infty\,,
 \]
 with 
 \[
 \|a\cdot \xi_{k} - \eta_{k}\cdot a\|\leq 2\|a\|\,\quad (k\in \bN)\,.
 \]
 Thus $K_{BASC}(\ell^{p})\leq 2$.  In fact  $K_{BASC}(\ell^{p})= 2$: 
 
 \end{example}
 \oomit{
 Now clause 3 of Theorem \ref{bddsemicont} has a consequence in this example. \marginpar{\bf \color{red}{changes here}}
 First note that, taking the $\|\cdot\|_{1}$ sum of the summands on the RHS,  we have the isometric isomorphism
  \[ 
 \ell^{p\#}\tensor \ell^{p\#} \simeq (\ell^{p}\tensor \ell^{p}) \oplus (e\otimes \ell^{p}) \oplus (\ell^{p}\otimes e) \oplus \bC (e\otimes e)\,.
  \]  
 
  So taking a general $M, N \in  \ell^{p\#}\tensor \ell^{p\#}$ with $\pi(M) =\pi(N) =e$,\marginpar{\bf \color{red}{omit this in view of 3.9?}}
  \[
 M] = e\otimes e +\sum_{m}b_{m}\delta_{m}\otimes e + \sum_{n} e\otimes c_{n}\delta_{n} + z\,,
 \]
  \[
 N  = e\otimes e +\sum_{m}b_{m}'\delta_{m}\otimes e + \sum_{n} e\otimes  c_{n}'\delta_{n} + z'\,,
 \]
for some $z, z' \in \ell^{p}\tensor \ell^{p}$. 
Then for each $k\in\bN$,
\[
 \delta_{k}\cdot M - N\cdot \delta_{k} =  (b_{k}+1) \delta_{k}\otimes e + e\otimes (c_{k}' +1)\delta_{k} + \cdots
\]
so that, since $\|\delta_{k}\| = 1$ for any $k$, 
\[
|b_{k}+1| + |c_{k}' +1| \leq  \|\delta_{k}\cdot M - N\cdot \delta_{k}\| = \|{\rm ad}_{M, N}(\delta_{k}) \| \leq  \|{\rm ad}_{M, N} \| \leq K_{BASC}(\ell^{p})\,.
\]
Letting $k\to \infty$ it follows that $2\leq K_{BASC(\ell^{p})}$.  Whence $K_{BASC(\ell^{p})} = 2$.

\end{example}
  
 Note that the argument works for any index set, but does not work in the weighted case.

\vskip 2mm
\oomit{{\color{red}{DOES NOT WORK SO FAR
Let $A$ be a unital  \BAC\ Banach algebra, and set
\begin{eqnarray}\label{new constant}
C(A) = \inf\{K:& \kern -1em{\rm \ there\  is\ } (\xi_{i}), (\eta_{i})\subset A\tensor A{\rm\ with\ } a\cdot \xi_{i} - \eta_{i} \cdot a \to 0,\\ 
&  \|a\cdot \xi_{i} - \eta_{i} \cdot a\|\leq K\|a\|,  a\in A, {\rm \ all\ } i\} > 0\,.
\end{eqnarray}
$
 Let $A_{m}$ be the algebra $A$ with the new (equivalent) norm  $\|\cdot\|_{m} = m\|\cdot\|$, for some $m > 1$.  Let $M, N\in A_{m} \tensor A_{m}$, $a\in A_{m}$.  Then 
\[
\|a\cdot M - N\cdot a\|_{m} = m^{2}\|a\cdot M - N\cdot a\|\,,
\]
whence $C(A_{m}) = mC(A)$.
%
Attempting the same with $K_{BASC}$ does not work as the norm on $A_{m}^{\#}\tensor A_{m}^{\#}$ is not $m^{2}$ times that on $A_{m}$.  This is true on elements of $A_{m}\tensor A_{m}$, but only $m$ times  on $e\tensor A_{m}$ and $A_{m}\otimes e$.
\vskip 1cm
Consider $c_{0}(A_{m})$.  Suppose this is \BASC.  Then there is $K>0$ and nets  $(\xi_i)$, $(N_i)$ in $c_{0}(A_{m})^{\#}\tensor  c_{0}(A_{m})^{\#}$ such that
\begin{enumerate}
\item  $a\cdot \xi_i - N_i \cdot a \to 0\,,\quad (a\in c_{0}(A_{m}))$\,,
\item for all $i$, $\|{\rm ad}_{\xi_{i},\eta_{i}}\|\leq K$\,,
\item $ \pi(\xi_i) = E$ and $ \pi(N_i) = E$\,.
\end{enumerate}
where $E$ is the adjoined identity.
Let $P_{k}: c_{0}(A_{m})^{\#}\to A_{k}: \lambda E + (f_{m})\mapsto \lambda e_{k} + f_{k}$. Note that $(f_{m})\mapsto f_{k}$ is of norm $1$,   $\|E\| = 1$ but $\|e_{k}\| = k$.  Consider $P_{k}\otimes P_{k}:c_{0}(A_{m})^{\#}\tensor  c_{0}(A_{m})^{\#} \to A_{k}\tensor A_{k}$.  Then for $a\in A_{k}$, unfortunately only bound seems to be
\[
\|a\cdot \big(P_{k}\otimes P_{k}\big)(\xi_{i}) -  \big(P_{k}\otimes P_{k}\big)(\eta_{i})\cdot a\|_{k} \leq K\|a\|_{k} k^{2}\,.
\]
}}
}
                    
}


\begin{proposition}\label{lowerbound}
Let $A$ be a Banach algebra which does not have an identity. Suppose that $\Phi_{A}$ is infinite and that
\begin{equation}\label{inf}
\inf\{\|\phi\| :\phi\in \Phi_{A}\} =\delta > 0\,.
\end{equation}
Then $K_{BASC}\geq 2\delta$\,.
\end{proposition}
 \begin{proof} For $\xi, \eta$ in $A^{\#}\tensor A^{\#}$ with $\pi(\xi) = \pi(\eta) = e$,  write 
\[ 
 \xi = (e\otimes e) + (F\otimes e)  +(e\otimes F' )+ \xi'\,, \eta = (e\otimes e) + (G'\otimes e ) +(e\otimes G) + \eta'\,,
 \]
where $F, F', G, G' \in A$, $\xi', \eta'\in A\tensor A$.  
 Then for $a\in A$,
\[
a\cdot \xi - \eta\cdot a =  (a + aF)\otimes e - e\otimes (a+Ga) + \zeta\,,
\]
for some $\zeta\in A\tensor A$.   Take $\varepsilon > 0$, and suppose that $A$  is \BASC. Since
\[
\|a\cdot \xi - \eta\cdot a\| \geq \|a + aF\| + \|a+Ga\|\,,
\]
by choosing $\xi$  and $\eta$ such that for any $a\in A$,

\[
(K_{BASC}(A)+\varepsilon)\|a\| \geq \|a\cdot \xi - \eta\cdot a\| \,,
\]
we have, for any $\varphi\in \Phi_{A}$,
\[
(K_{BASC}(A)+\varepsilon)\|a\| \geq  \|a + aF\| + \|a+Ga\| \geq |\phi(a)|\big(| 1+\phi(F)| + |1 +\phi(G)|\big)\,.
\]
Take $0 <\rho <1$.  By non-compactness of $\Phi_{A}$, there is a net $(\phi_{\alpha})$ going to $\infty$ in $\Phi_{A}$.  Take norm one elements $a_{\alpha}\in A$ with $\phi_{\alpha}(a_{\alpha}) \geq \rho\delta$.  Then 
\[
K_{BASC}(A)+\varepsilon\geq \rho \delta\big(| 1+\phi_{\alpha}(F)| + |1 +\phi_{\alpha}(G)|\big)\,.
\]
Taking the limit over $\alpha$, $K_{BASC}(A)+\varepsilon\geq 2\rho\delta$, for each $0 < \rho < 1$, each $\varepsilon > 0$.  It follows that  $K_{BASC}(A)\geq 2\delta$.
\end{proof}
Note that if $A$ has a bounded approximate identity of norm $m$, then $\delta \geq 1/m$ in the above.

A sufficient condition for \eqref{inf} to hold is the following. The hypothesis here is certainly not necessary -- in particular it excludes $\ell^{p}$.  

\begin{proposition}\label{lower bound}
Let $A$ be a separable Banach algebra with $A^{2} = A$ and $\Phi_{A}\not = \emptyset$.  Then $\inf\{\|\phi\| :\phi\in \Phi_{A}\} > 0$.
\end{proposition}

\begin{proof}
By \cite[Theorem 1.3]{Loy} there are $n\in \bN$ and $M > 0$ such that for  $x\in A$, there exist $x_{1}, \ldots, x_{n} \in A$ such that
 \[
  x = \sum_{i=1}^{n} x_{i} y_{i}\,, \quad \sum_{i=1}^{n} \|x_{i}\| \|y_{i}\| \leq M\|x\|\,.
 \]
  Thus for any $\phi\in \Phi_{A}$,
 \[ 
 |\phi(x)| \leq \sum_{i=1}^{n} |\phi(x_{i})| |\phi( y_{i})| \leq \|\phi\|^{2}  \sum_{i=1}^{n} \|x_{i}\| \|y_{i}\| \leq M\|\phi\|^{2}\|x\|\,.
  \] 
  It is immediate that $\|\phi\|\geq M^{-1}$.
  \end{proof}
  
\oomit{
Now consider $c_{0}(A_{m})$, \ApsC\ by Theorem \ref{sums}. \marginpar{\bf \vskip -1cm \color{blue}{this paragraph harks back to $A_{m}= A$ with $\|\cdot\|_{m} =m\|\cdot\|$, which got nowhere.  It should be removed.}}
Indeed, if $K_{BASC}(A_{m})$ are bounded, \cite[Lemma 3.2]{BEJ} shows  $c_{0}(A_{m})$ is \BASC.  $c_{0}(A_{m})$ contains $A_{m}$ as a summand, and so cannot be \BASC\ if $\{K_{BASC}(A_{m})\}$ is unbounded, since $A_{m}\tensor A_{m}$ sits naturally as a summand  in $c_{0}(A_{m})\tensor c_{0}(A_{m})$.  More generally, we have
}

\begin{proposition}\label{quotients}
Let $A$ be an \ApsC\ Banach  algebra, $\Theta : A\to B$  a continuous epimorphism.  Then $B$ is \ApsC.  If $A$ is \BASC\ then so is $B$.  If, further, $\Theta$ is the quotient map determined by a closed ideal, then
\[
K_{BASC}(B) \leq K_{BASC}(A)\,.
\]
The same holds for \BApsA.
\end{proposition}

\begin{proof}  Let $D:B\to X$ be a continuous derivation into a Banach $B$-bimodule $X$. Set $Y = X$ with $A$ actions $a\bullet y = \Theta(a)\cdot y$,  $y \bullet a = y \cdot \Theta(a)$.  Then $D\circ \Theta : A\to Y$ is a continuous derivation and so there exist nets $(\xi_{i})$, $(\eta_{i})$ in $Y$ such that for $a\in A$,
\[
D\circ\Theta(a) = \lim_{i}(a\bullet \xi_{i} - \eta_{i}\bullet a) =  \lim_{i} (\Theta(a)\cdot \xi_{i} - \eta_{i} \cdot \Theta(a))\,.
\]
Since $\Theta$ is surjective, $D$ is approximately semi-inner.

In the bounded case, given $\varepsilon >0$,  we can choose $(\xi_{i})$ and $(\eta_{i})$ such that for all $i$, $\|a\bullet \xi_{i} - \eta_{i}\bullet a)\|\leq (K_{BASC}(A) +\varepsilon)|a\|$, $a\in A$. That is, for $a\in A$,
\[
\|\Theta(a)\cdot \xi_{i} - \eta_{i}\cdot \Theta(a)\|\leq (K_{BASC}(A) +\varepsilon)\|a\|\,.
\]
But then for $b\in B$, and $a\in \Theta^{-1}(b)$,
\[
\|b\cdot \xi_{i} - \eta_{i}\cdot b\|\leq (K_{BASC}(A) +\varepsilon)\|a\|\,.
\]
By the open mapping theorem there is a constant $K' > 0$  such that \\ $\inf\{\|a\| :a\in \Theta^{-1}(b)\}\leq K'\|b\|$\,. Whence
\[
\|b\cdot \xi_{i} - \eta_{i}\cdot b\|\leq (K_{BASC}(A) +\varepsilon)K'\|b\|\,.
\]
Thus $B$ is \BASC.  

The boundedly semi-amenable case is similar. In the quotient case $K' = 1$.
\end{proof}

 In the case of a SIN-group, the next result follows from Theorem \ref{semi+baa=approx},  or from Theorem \ref{Segal}.
 
 \begin{theorem}\label{L1case}
 For a locally compact $G$, $L^{1}(G)$ is \ApsA\ if and only if $G$ is amenable.
 \end{theorem}
 
 \begin{proof}
 Suppose that $L^{1}{G}$ is \ApsA\ and follow the classical argument  as in \cite[Theorem 3.2]{GhaLoy}.  The derivation $\Delta$ there is approximately semi-inner, giving nets $(\xi_{i}), (\eta_{i}) \subset Y^{*}$ such that
 \[
 f\cdot \delta_{e} - \delta_{e} \int_{G} f = \lim_{i} \Big( f\cdot \xi_{i} - \eta_{i} \int_{G} f\Big) \quad f\in L^{1}(G)\,.
 \]
 Taking $\phi\in L^{1}(G)$, $\phi\geq 0$, $\|\phi\| = 1$ gives
 \begin{eqnarray*}
 \Delta(\delta_{g}) &=& \Delta(\delta_{g})\cdot \phi = \Delta(\delta_{g} * \phi) -\delta_{g} \cdot \Delta(\phi)\\
 &=& \lim_{i}((\delta_{g} * \phi)\cdot \xi_{i} - \eta_{i}-\delta_{g} \cdot (\phi * \xi_{i} - \eta_{i})\\
 &=& \lim_{i} (\delta_{g} \cdot \eta_{i} - \eta_{i})\,.
 \end{eqnarray*}
 
 Now continue the argument of \cite[Theorem 3.2]{GhaLoy}, with $(\eta_{i})$ in place of $(\xi_{i})$, to get $G$ amenable.
 \end{proof}
 
 \begin{theorem}
   $L^{1}(G)^{**}$ is \ApsA\ if and only if $G$ is finite.
 \end{theorem} 
 \begin{proof}
 $L^{1}(G)^{**}$ has a right identity, and the hypothesis means it has left and right approximate identities. Thus it has an identity.  But then it is \ApA, Corollary \ref{semi+baa=approx}, so $G$ is finite by \cite[Theorem 3.3]{GhaLoy}
 \end{proof}

\section{Examples}
\subsection{Segal algebras}

For a locally compact group $G$ a subspace $S(G)$ of $L^{1}(G)$ is a Segal algebra if it satisfies the following:
\begin{itemize}
\item[$(S_{0})$]  $S(G)$ is properly dense in $L^{1}(G)$\,
\item[$(S_{1})$]  $S(G)$ is a Banach space under some norm $\|\cdot\|_{S(G)}$ dominating $\|\cdot \|_{1}$\,,
\item[$(S_{2})$]  $S(G)$ is left translation invariant\,,
\item[$(S_{3})$]  the norm $\|\cdot\|_{S(G)}$ is left invariant.
\end{itemize}
These refer to a Banach space, but it follows that $S(G)$ is a Banach algebra under $\|\cdot \|_{S(G)}$. Indeed, $S(G)$ is a left ideal of $L^{1}(G)$, and  $\|h * f\|_{S(G)} \leq  \|h\|_{1} \|f\|_{S(G)}$, $h\in L^{1}(G), f\in S(G)$.  See \cite[A3]{ReiSteg} for further details.

Recall that a locally compact group $G$ is called SIN if it has a basis of compact neighbourhoods of the identity each of which is invariant under all the inner automorphisms of the group.  This is equivalent to $L^{1}(G)$ having a central \BAI\ 
\cite[Theorem 1(b) and Remark]{KaR}.

\begin{theorem}\label{Segal} 
Let $G$ be a SIN group.  Then the following are equivalent:
\begin{enumerate}
\item  There exists a Segal subalgebra of $L^{1}(G)$ that is \ApsA.
\item  Every  Segal subalgebra of $L^{1}(G)$ is \ApsA.
\item $G$ is amenable.
\end{enumerate}
\end{theorem}

\begin{proof}  (1) $\Rightarrow$ (3).  
Suppose that the Segal subalgebra $S(G)$ is \ApsA.  The closed ideal $\{ f\in S(G) : \int_{G} f(x) dx = 0\}$  is clearly complemented in $S(G)$, so by Theorem \ref{proj} has a right approximate identity. But then by \cite [Proposition 5.4 (i)]{CGZ} the augmentation ideal $\{ f\in L^{1}(G) : \int_{G}f(x) dx = 0\}$ has a right approximate identity. So by \cite[Theorem 5.2]{Willis}, $G$ is amenable.

(3) $\Rightarrow$ (2).
Let $G$ be an amenable group, let $X$ be a Banach $S(G)$-bimodule, and $D:S(G)\to X^{*}$ a continuous derivation. 
By \cite[Theorem 1(b) and Remark]{KaR} there is
 an approximate identity of $S(G)$, $(e_{i})$,  that is a central and bounded in $L^{1}(G)$. 
  Take elements $g, h\in G$. Then
\begin{equation}\label{one}
D(\delta_{gh} * e_{i}^{2}) =D(\delta_{g} * e_{i})\cdot  (\delta_{h} * e_{i}) +  (\delta_{g} * e_{i})\cdot D( \delta_{h} * e_{i})\,.
\end{equation}
 Letting $\delta_{(gh)^{-1}} * e_{i}^{2}$ act on the left of \eqref{one}, and noting that $(e_{i})$ is central,  we find
 \begin{eqnarray}\label{two}
( \delta_{(gh)^{-1}} * e_{i}^{2}) \cdot D(\delta_{gh} * e_{i}^{2}) &=&  (\delta_{h^{-1}} * e_{i})   * ( \delta_{g^{-1}} * e_{i}) \cdot D(\delta_{g} * e_{i})\cdot  (\delta_{h} * e_{i})\nonumber \\
& &\qquad + (\delta_{h^{-1}} * e_{i}^{3})\cdot D( \delta_{h} * e_{i})\,.
 \end{eqnarray}
 
 Now $S(G) = L^{1}(G) * S(G)$ by the Cohen factorization theorem for modules, \cite[ VIII.32.50]{HewRoss}, so given
 $\varphi\in  S(G)$, $\varphi = h * \vartheta$ for some $h\in L^{1}(G)$, $\vartheta\in S(G)$.  Since $t\mapsto \delta_{t} * h$ from $G$ to $L^{1}(G)$ is left-uniformly continuous, the same is true for the map $t \mapsto \delta_{t} * \varphi$  from $G$ to $S(G)$.  Thus for each $x\in X$, and each index $i$,
  the functions
  \begin{eqnarray*}
  t \mapsto \langle \delta_{t^{-1}} * e_{i}^{2})\cdot D( \delta_{t} * e_{i}^{2}), x\rangle\,,\\
 s \mapsto \langle \delta_{s^{-1}} * e_{i})\cdot D( \delta_{s} * e_{i}), x\rangle\,, 
 \end{eqnarray*}
 are bounded and left-uniformly continuous on $G$.
 
 Let $M$ be an invariant mean on the space 
 of bounded left-uniformly continuous functions on $G$.  Define functionals $x_{i}^{*}, y_{i}^{*}\in X^{*}$ by
 \begin{eqnarray*}
 \langle x_{i}^{*}, x\rangle &=& \langle M,  t \mapsto \langle \delta_{t^{-1}} * e_{i}^{2})\cdot D( \delta_{t} * e_{i}^{2}), x\rangle \rangle\,,\\
 \langle y_{i}^{*}, x \rangle &=& \langle M,  s \mapsto \langle \delta_{s^{-1}} * e_{i})\cdot D( \delta_{s} * e_{i}), x\rangle \rangle\,.
 \end{eqnarray*}
  
  Then for $x\in X$, for any $s\in G$,  invariance of $M$  gives
  \begin{eqnarray*}
\langle x_{i}^{*}, x\rangle &=& \langle M,   t \mapsto \langle \delta_{t^{-1}} * e_{i}^{2})\cdot D( \delta_{t} * e_{i}^{2}), x\rangle \rangle \nonumber\\
 &=& \langle M,  t \mapsto  \langle(\delta_{(ts)^{-1}} * e_{i}^{2}) \cdot D(\delta_{ts} * e_{i}^{2}), x\rangle\rangle \,.
  \end{eqnarray*}
  So by \eqref{two},
  \begin{eqnarray}\label{thirteen}
 \langle  x_{i}^{*}, x\rangle&=& \langle M, t \mapsto \langle (\delta_{s^{-1}} * e_{i}) * ( \delta_{t^{-1}} *e_{i}) \cdot D(\delta_{t} * e_{i})\cdot  (\delta_{s} * e_{i})), x\rangle \rangle\nonumber\\
 & & \qquad + \langle (M\mapsto \langle(\delta_{s^{-1}} * e_{i}^{3})\cdot D( \delta_{s} * e_{i}), x \rangle\rangle\,.
  \end{eqnarray}
 
 Now
 \begin{eqnarray*}
 & & \kern -1cm \langle M, t \mapsto \langle (\delta_{s^{-1}} * e_{i})   * (\delta_{t^{-1}} * e_{i}) \cdot D(\delta_{t} * e_{i})\cdot  (\delta_{s} * e_{i}), x\rangle \rangle\nonumber\\
 &=& \langle M, t \mapsto(\delta_{t^{-1}}* e_{i}) \cdot D(\delta_{t} * e_{i}), (\delta_{s} * e_{i})\cdot x *(\delta_{s^{-1}} * e_{i}) \rangle \rangle\nonumber\\
  &=& \langle y_{i}^{*}, \langle (\delta_{s} * e_{i})\cdot x *(\delta_{s^{-1}} * e_{i}) \rangle\nonumber\\
 &=& \langle (\delta_{s^{-1}} * e_{i}) \cdot y_{i}^{*}\cdot  (\delta_{s} * e_{i}), x \rangle\,.
\end{eqnarray*}
 
Thus \eqref{thirteen} becomes
 \begin{eqnarray*}
  \langle x_{i}^{*}, x\rangle&=& \langle (\delta_{s^{-1}} * e_{i}) \cdot y_{i}^{*}\cdot  (\delta_{s} * e_{i}), x \rangle\nonumber\\
 & &\qquad+\langle (\delta_{s^{-1}} * e_{i}^{3})\cdot D( \delta_{s} * e_{i}), x\rangle\,.
  \end{eqnarray*}
   That is,
   \begin{equation}
x_{i}^{*} = (\delta_{s^{-1}} * e_{i}) \cdot y_{i}^{*}\cdot  (\delta_{s} * e_{i}) + (\delta_{s^{-1}} * e_{i}^{3})\cdot D( \delta_{s} * e_{i})\,. \nonumber
   \end{equation}
   
   Taking the action  on the left by $\delta_{s} * e_{i}$, and recalling the centrality of $e_{i}$,
   \begin{eqnarray*}
    \delta_{s} * e_{i}\cdot x_{i}^{*} &=& e_{i}^{2} \cdot y_{i}^{*}\cdot  (\delta_{s} * e_{i}) + e_{i}^{4}\cdot D( \delta_{s} * e_{i})\nonumber\\
    &=& e_{i}^{2} \cdot y_{i}^{*}\cdot  (e_{i} * \delta_{s}) +D( \delta_{s} * e_{i}^{5}) - D(e_{i}^{4}) \cdot (e_{i} * \delta_{s})\,.
   \end{eqnarray*}
 
 Hence,
 \begin{equation}\label{fourteen}
 D( \delta_{s} * e_{i}^{5}) =   \delta_{s}\cdot (e_{i}\cdot x_{i}^{*})  - (e_{i}^{2} \cdot y_{i}^{*}\cdot e_{i} - D(e_{i}^{4})\cdot e_{i} )\cdot \delta_{s} ) \,.
 \end{equation}
 
 It is immediate that \eqref{fourteen}  holds for any measure $\mu$ on $G$ which is a finite linear combination of point masses. Then since
the discrete measures of finite support are dense in $M(G)$ in the strong operator topology of $M(G)$ acting on $S(G)$,
\eqref{fourteen} holds for any $\mu = \varphi\in S(G)$.  Hence we can write
\begin{equation}\label{twentyfive}
D(\varphi * e_{i}^{5}) = \varphi \cdot \xi_{i} - \eta_{i}\cdot \varphi\,,\quad (\varphi\in S(G))\,,
\end{equation}
where $\xi_{i} = e_{i} \cdot x^{*}_{i} $ and $\eta_{i} =  (e_{i}^{2}\cdot y^{*}_{i} - D(e_{i}^{4}))\cdot e_{i}$.  

For $\varphi\in S(G)$, certainly  $\|\varphi*e_i - \varphi\|_{S(G)}\stackrel{i}{\to} 0$.  Thus
\begin{eqnarray*}
\|\varphi*e_{i}^5 - \varphi\|_{S(G)} &\leq& \sum_{k=1}^4\|(\varphi*e_i - \varphi)*e_i^k\|_{S(G)}+\|\varphi*e_i - \varphi\|_{S(G)}\\
&\leq&  \sum_{k=1}^4\|(\varphi*e_i - \varphi)\|_{S(G)} \|e_i^k\|_1+\|\varphi*e_i - \varphi\|_{S(G)}\stackrel{i}{\to} 0\,.
\end{eqnarray*}
Here we have used the $\|\cdot\|_1$-boundedness of $(e_i)$.  
Taking the limit over $i$ in \eqref{twentyfive}, it follows that
$$
D(\varphi) = \lim_{i} ( \varphi\cdot \xi_{i} - \eta_{i}\cdot \varphi) \qquad (\varphi\in S(G)\,,
$$
so that $D$ is approximately semi-inner. 

(2) $\Rightarrow$ (1) is trivial.
  \end{proof}
   
 \begin{remark}
 A non-trivial Segal algebra $S(G)$ is never \ApA, \cite{Alag}. See also \cite[\S4]{DL}, \cite[\S 3]{CG}.  Theorem \ref{Segal} gives the SIN group case of Theorem \ref{L1case}.
 \end{remark}


\subsection{Banach function algebras}
\oomit{

\begin{theorem}
Let $(A, \|\cdot\|_A)$ be a Banach algebra with norm $\|\cdot\|_A$, $J$ a proper (dense) ideal in $A$ which is a Banach algebra under its own norm $\|\cdot\|_J$ with $\|\cdot \|_{A}\leq \|\cdot \|_J$, and such that $\|ab\|_J \leq \|a\|_A \|b\|_J$ , $a\in A, b\in J$.  Suppose further that $(A, \|\cdot\|_A)$ has a central bounded approximate identity $(e_\alpha)\subset J$, and that $(e_\alpha)$ is an approximate identity in $(J, \|\cdot \|_J)$.  Then if $A$ is amenable, $J$ is pseudo-amenable.
\end{theorem}
\begin{remark}  Typical examples are Segal algebras over amenable SIN groups.
   Note that $(e_\alpha)$ cannot be  $\|\cdot\|_J$ - bounded.  The argument does not work for Schatten $p$-classes in $\mathcal{K}(H)$ because of lack of  a \emph{central} approximate identity. 
\end{remark}
\begin{proof} Let $m\in A\tensor A$, so that
$$
m = \sum_k a_k\otimes b_k \quad {\rm with }\quad \sum_k \|a_k\|_A\,\|b_k\|_A < \infty\,.
$$
Now for $ c, d\in J$,
$$
  \|ca_k\|_J \leq \|c\|_J \|a_k\|_A\,\ {\rm and }\ \  \|b_k d\|_J \leq \|d\|_J \|b_k\|_A\,, 
 $$
 so that
 $$
 \sum_k c a_k \otimes b_k d
  $$
is convergent in $J\tensor J$. Furthermore,
$$
\|c\cdot m \cdot d\|_{J\tensor J} = \|\sum_k c a_k \otimes b_k d\|_{J\tensor J}\leq \|c\|_J\|d\|_J \| m \|_{A\tensor A}\,.
$$
Take  a (bounded) approximate diagonal $(\xi_i)\subset A\tensor A$  for $A$.  Write 
$$
M_\lambda = \sum_k a_k^\lambda\otimes b_k^\lambda \,,\quad  \sum_k \|a^\lambda_k\|_A\,\|b^\lambda_k\|_A <\infty\,.
$$
Take $b\in J$.  Then 
\begin{eqnarray*}
& \|be_\alpha \cdot M_\lambda \cdot e_\alpha  - e_\alpha\cdot M_\lambda \cdot e_\alpha b \|_{J\tensor J}=
\| e_\alpha b \cdot M_\lambda \cdot e_\alpha - e_\alpha\cdot M_\lambda \cdot b e_\alpha \|_{J\tensor J}\\ \\
& \leq \| e_\alpha\|^2_J \|b\cdot M_\lambda - M_\lambda\cdot b\|_{A\tensor A}\,.
\end{eqnarray*}
%
%
We also have
\begin{eqnarray*}
&&\pi(e_\alpha \cdot M_\lambda \cdot e_\alpha) b - b = e_\alpha \pi(M_\lambda)e_\alpha b - b \\
&&= e_\alpha (\pi(M_\lambda) b - b)e_{\alpha} + e_\alpha (be_{\alpha} - b)+ (e_{\alpha} b - b)\,,
\end{eqnarray*}
so that 
$$
\|\pi(e_\alpha \cdot M_\lambda \cdot e_\alpha) b - b \|_{J} \leq \|e_{\alpha}\|_{J}^{2} \| \pi(M_\lambda) b - b\|_{A}
+(\|e_{\alpha}\|_{A}+1)\|be_{\alpha} - b\|_{J}\,.
$$

Set $K = 1+ \sup_{\alpha, \lambda} \|e_\alpha\|_A\|M_\lambda\|_{A\tensor A}$, take $b_1, b_2, \ldots, b_n\in J$, and take $\varepsilon > 0$.  Choose $\alpha$ such that each $\|b_j e_\alpha - b_j\|_A < \varepsilon /(2nK)$. For this $\alpha$,  choose $\lambda$ such that   each $\|b_j \cdot M_\lambda - M_\lambda \cdot b_j\|_{A\tensor A}\leq \varepsilon/(K\|e_\alpha\|_J^2)$, and $\|\pi(M_\lambda)b_j - b_j\|_A < \varepsilon/(2K\|e_\alpha\|_J)$.  We then have, for each $b_{j}$,
\begin{eqnarray*}
&&\|b_j\cdot (e_\alpha\cdot M_\lambda\cdot e_\alpha) - (e_\alpha\cdot M_\lambda\cdot e_\alpha)\cdot  b_{j}\|_{J\tensor J } < \varepsilon\,,\\
&&\hskip 1.3cm \|\pi(e_\alpha \cdot M_\lambda \cdot e_\alpha) b_j - b_j \|_J < \varepsilon\,.
\end{eqnarray*}
The existence of an approximate diagonal in $J$ is now standard.
\end{proof}

It is of interest to note exactly how non-centrality causes a difficulty.
Take  a (bounded) approximate diagonal $(\xi_i)\subset A\tensor A$  for $A$.  Write 
$$
M_\lambda = \sum_k a_k^\lambda\otimes b_k^\lambda \,,\quad  \sum_k \|a^\lambda_k\|_A\,\|b^\lambda_k\|_A <\infty\,.
$$
Looking to show  that  $(e_{\alpha} \cdot M_{\lambda} \cdot e_{\alpha})$ gives an approximate diagonal in $J$, take $b\in J$.  Then
\begin{eqnarray*}
&&\kern -1cm \pi(e_\alpha \cdot M_\lambda \cdot e_\alpha) b - b = e_\alpha \pi(M_\lambda)e_\alpha b - b\\
&&= e_\alpha (\pi(M_\lambda) b - b)e_{\alpha} + e_\alpha (be_{\alpha} - b)+ (e_{\alpha} b - b)-e_\alpha\pi(M_\lambda)(be_\alpha - e_\alpha b)\,,
\end{eqnarray*}
so that 
\begin{eqnarray*}
\|\pi(e_\alpha \cdot M_\lambda \cdot e_\alpha) b - b \|_{J} &\leq& \|e_{\alpha}\|_{J}^{2} \| \pi(M_\lambda) b - b\|_{A}
+\|e_{\alpha}\|_{A}\|be_{\alpha} - b\|_{J}\\
&+& \|e_{\alpha}b - b\|_{J}+ \|e_\alpha\|_A \| \pi(M_\lambda)\|_{A} \|be_\alpha - e_\alpha b\|_J\,.
\end{eqnarray*}
So no difficulties arise here,  first choose $\alpha$ so that second and third terms are small, then, for this $\alpha$, choose $\lambda$ so first term is small.

However,
\begin{eqnarray*}
&& be_\alpha \cdot M_\lambda \cdot e_\alpha  - e_\alpha\cdot M_\lambda \cdot e_\alpha b\\
&& =
e_\alpha (b \cdot M_\lambda - M_\lambda \cdot b) e_\alpha
+(be_\alpha  - e_\alpha b)\cdot M_\lambda\cdot  e_\alpha + e_\alpha\cdot M_\lambda\cdot (be_\alpha  - e_\alpha b)\,. 
\end{eqnarray*}
Thus we have the following estimate
\begin{eqnarray*}
 &&\kern -0.5cm \|be_\alpha \cdot M_\lambda \cdot e_\alpha  - e_\alpha\cdot M_\lambda \cdot e_\alpha b \|_{J\tensor J}\\
 && \phantom{iyrtrug}\leq \| e_\alpha\|^2_J \|b\cdot M_\lambda - M_\lambda\cdot b\|_{A\tensor A}+ 2\|be_\alpha  - e_\alpha b\|_J \|e_{\alpha}\|_J \|M_\lambda\|_{A\tensor A}\,.
 \end{eqnarray*}
It is the last term  $\|be_\alpha  - e_\alpha b\|_J \|e_{\alpha}\|_J$  that is the difficulty. If $(e_\alpha)$ was quasi-bounded in $J$, then the argument would proceed as before. But this is a most unlikely scenario, and certainly fails in the case of Schatten classes, \cite[\S1.9]{DL}.

\begin{example}\label{ellp}
For $1\leq p < \infty$, the algebra $\ell^p$ under pointwise operations is not \ApA, \cite{DLZ, CG}.  However, derivations from $\ell^p$ are always approximately semi-inner, \cite[Example 3.7]{GhaLoy2}.  
\end{example}
}

\begin{theorem}\label{BFAG}
Let $A$ be a Banach function algebra on a discrete set $S$. 
Suppose that $A$ has  an approximate identity of elements of finite support. Then
\begin{enumerate}
\item[\rm(i)] $A$ is approximately semi-contractible.
\item[\rm(ii)]  if the \AI\ is multiplier bounded then $A$ is \BASC.
\item[\rm(iii)]  if the \AI\ is quasi-bounded then $A$ is \AC.
\end{enumerate}
\end{theorem}

\begin{proof}
(i)  Let $(e_{\alpha})$ be the approximate identity, with $S_{\alpha} = {\rm supp} (e_{\alpha})$.  Define
\begin{eqnarray*}
\xi_{\alpha} &=& e\otimes e - e_{\alpha}\otimes e + \sum_{i\in S_{\alpha}} e_{\alpha}\delta_{i} \otimes \delta_{i}\,,\\
\eta_{\alpha} &=& e\otimes e - e\otimes e_{\alpha}+ \sum_{i\in S_{\alpha}} e_{\alpha}\delta_{i} \otimes \delta_{i}\,.\\
\end{eqnarray*}
Then $\pi(\xi_{\alpha}) = \pi(\eta_{\alpha})= e$, and
\[
a\cdot \xi_{\alpha} - \eta_{\alpha}\cdot a = (a-e_{\alpha}\otimes e +e\otimes (a-ae_{\alpha}) \to 0\,.
\]
(ii)  We have
\[
\|a\cdot \xi_{\alpha} -\eta_{\alpha}\cdot a \| = 2\|a-e_{\alpha}a\|\leq 2(1+K)\|a\|\,,
\]
if  $\|e_{\alpha} a\|\leq K\|a\|$, $a\in A$.
\oomit{
(i) Let $(e_{\alpha})$ be an approximate identity lying in $c_{00}(S)$. Let $D:A\to X$ be a continuous derivation into a Banach $A$-bimodule $X$.   Define $n_{\alpha} = |{\rm supp}(e_{\alpha})|$. Then $e_{\alpha}A\simeq \mathbb C^{n_{\alpha}}$, so $D|_{e_{\alpha}A}$ is inner,
 say implemented by $\xi_{\alpha}\in X$. Thus for $A$, 
\begin{eqnarray*}
D(a) &=& \lim{i}m_{\alpha} D(e_{\alpha}a) = \lim{i}m_{\alpha}\Big(e_{\alpha}a\cdot\xi_{\alpha} - \xi_{\alpha}\cdot e_{\alpha}a\Big)\\
& =& \lim{i}m_{\alpha} \Big(a\cdot(e_{\alpha}\cdot \xi_{\alpha}) - (\xi_{\alpha}\cdot e_{\alpha})\cdot a\Big)\,.
\end{eqnarray*}
The standard method of taking finite subsets gives the required net to show that $D$ is approximately semi-inner.
}

(iii) is \cite[Proposition 4.2]{GhaSt}.
\end{proof}


\begin{remark}
Note that in (iii), $2(1+K)$ gives an upper bound of $4$ for $K_{BASC}(J_{p})$, the comment before Theorem \ref{lower bound} gives lower bound $1$.
\end{remark}
\begin{corollary}\label{F2}
 For each $n\geq 2$, the algebra $A(\mathbb F_{n})$ is \BASC\ but not \ApA.
\end{corollary}
\begin{proof}
By \cite[Theorem 2.1]{Haag}, $A(\mathbb F_{n})$ has a muliplier-bounded approximate identity of functions of compact support.  Theorem \ref{BFAG}(ii) gives bounded approximate semi-contractiblity. Failure of approximate amenability is \cite[\S4]{CGZ}.
 \end{proof}

\begin{remark}  While a \BAC\ algebra always has a bounded approximate identity \cite[Corollary 3.4]{CGZ}, this is not the case for bounded approximate semi-amenability.  For example, the Fourier algebra $A(G)$ of a locally compact group has a \BAI\ if and only if $G$ is amenable \cite{Lep}.
\end{remark}
\begin{corollary}\label{maxmin}
  The algebra $\ell^{1}(\bN_{\wedge},\omega)$ is always approximately semi-amenable, and is \ApA\ if and only if $\liminf_{n\to\infty} \omega_{n} <\infty$.  The algebra $\ell^{1}(\bN_{\vee},\omega)$ is approximately (semi)-amenable if and only if $\liminf_{n\to\infty} \omega_{n} <\infty$.
 \end{corollary}
\begin{proof}
   $\ell^{1}(\bN_{\wedge},\omega)$   is isomorphic to the Feinstein algebra ${\mathcal A}_{\omega}$, \cite[3.2]{DL}.

 By Proposition \ref{cbaistrong},  being \ApA\ and \ApsA\ are the same for $\ell^{1}(\bN_{\vee},\omega)$. Supposing, then, that   $\ell^{1}(\bN_{\vee},\omega)$ is \ApA,  \cite[3.11]{DL}, gives  $\ell^{1}(\bN_{\wedge},\omega)$ is \ApA, and so we have 
 {$\liminf_{n\to\infty} \omega_{n} <\infty$}  by \cite[Theorem 3.10]{DL}. Conversely, if this limit is finite, then the calculation (for a suitable subsequence) at the end of \cite[Example 4.6]{GhaLoyZh} shows that $\ell^{1}(\bN_{\vee},\omega)$ is \ApA.
 
     \end{proof}
\begin{remark}
It is shown in \cite[3.2]{DL} that  the maximal ideal $M$ at $\infty$ in  $\ell^{1}(\bN_{\vee},\omega)$ is isomorphic to 
 $\ell^{1}(\bN_{\wedge},\sigma)$ where $\sigma_{j} = \omega_{j+1}$.  Thus $M$ is always approximately semi-amenable, and \ApA\  if and only if  $\ell^{1}(\bN_{\wedge},\omega)$ is, if and only  if $\liminf_{n\to\infty} \omega_{n} <\infty$.
\end{remark}

  \section{Sums and products of \ApsA\ algebras}
  
 It is known that if $A$ and $B$ are \ApA\ then in general $A\oplus B$ can fail to be \ApA, even with $B = A^{\rm op}$, \cite{GhaR1}. The situation for $A\oplus A$ is open for general $A$.  
 
The situation for approximate semi-amenability is rather better.
  
  \begin{theorem}\label{directsums} 
(i)  Suppose that $A$ and $B$ are \ApsA\ Banach algebras. Then $A\oplus B$ is  \ApsA\ .\\
 (ii) Suppose that $A$ and $B$ are boundedly approximately semi-amenable (-contractible) Banach algebras, and each has a central multiplier bounded approximate identity. Then  $A\oplus B$ has the same properties.
 \end{theorem} 
 
\begin{proof}  (i) By Theorem \ref{parallel}, $A$ has a left $(e_{\alpha})$ and a right $(f_{\beta})$ approximate identity, as does $B$, $(g_{\gamma})$ and $(h_{\delta})$. 

By Theorem \ref{equiv}  \ApsA\ is the same as \ApsC. 
Let  $X$ be a  Banach $(A\oplus B)$-bimodule, $D:A\oplus B \to X$ be a continuous derivation.  Take $\varepsilon > 0$ and $a_{1}, 
\ldots a_{k}\in A$, $b_{1}, \ldots b_{k}\in B$. Considering $D|_{A}$ and $D|_{B}$,  there are $x, y, w, z\in X$ such that
\[
\|D(a_{i}\oplus 0) - (a_{i}\oplus 0)\cdot x + y\cdot (a_{i}\oplus 0)\| < \varepsilon/2 \quad (i=1, \ldots, k)\,,
\]
and also $w, z\in X$ such that
\[
\|D(0\oplus b_{i}) - (0\oplus b_{i})\cdot w + z\cdot (0\oplus b_{i})\| < \varepsilon/2 \quad (i=1, \ldots, k)\,.
\]

Then there is $\alpha, \beta, \gamma, \delta$ such that
\begin{equation}\label{a}
\|D(a_{i}\oplus 0) - (a_{i}f_{\beta}\oplus 0)\cdot x + y\cdot (e_{\alpha}a_{i}\oplus 0)\| < \varepsilon/2 \quad (i=1, \ldots, k)\,,
\end{equation}
and 
\begin{equation}\label{b}
\|D(0\oplus b_{i}) - (0\oplus b_{i}h_{\delta})\cdot w + z\cdot (0\oplus g_{\gamma}b_{i})\| < \varepsilon/2 \quad (i=1, \ldots, k)\,.
\end{equation}

Now \eqref{a} can be rewritten as
\begin{equation}\label{firstsummand}
\|D(a_{i}\oplus 0) - (a_{i}\oplus b_{i})\cdot((f_{\beta}\oplus 0)\cdot x) + (y\cdot (e_{\alpha}\oplus 0))\cdot(a_{i}\oplus b_{i})\| < \varepsilon/2 \quad (i=1, \ldots, k)\,,
\end{equation}
and \eqref{b} as
\begin{equation}\label{secondsummand}
\|D(0\oplus b_{i}) - ((a_{i}\oplus b_{i}))\cdot(0\oplus h_{\delta})\cdot w+ (z\cdot (0\oplus g_{\gamma}))\cdot(a_{i}\oplus b_{i})\| < \varepsilon/2 \quad (i=1, \ldots, k)\,.
\end{equation}
It follows that
\begin{eqnarray*}
\|D(a_{i}\oplus b_{i})& - &(a_{i}\oplus b_{i})\cdot \overbrace{\Big( f_{\beta}\oplus 0)\cdot x +(0\oplus h_{\delta})\cdot w\Big)}^{\xi_{\lambda}}\\
&&\kern 2cm +\ \overbrace{\Big(y\cdot (e_{\alpha}\oplus 0) + z\cdot (0\oplus g_{\gamma}\Big)}^{\eta_{\lambda}}\cdot(a_{i}\oplus b_{i})\| < \varepsilon
\end{eqnarray*}
for $i=1, \ldots, k$.   Taking the directed set $\Lambda = \cF(A\times B)\times (0, \infty)$ with partial order $(F, \varepsilon)\leq (G, \rho)$ defined as $F\subset G$ and $\rho \leq \varepsilon$, it is immediate that there are nets $(x_{\lambda})_{\Lambda}, (\eta_{\lambda})\subset X$ with
\[
D(a\oplus b) = \lim_{\lambda}((a\oplus b)\cdot\xi_{\lambda} - \eta_{\lambda}\cdot (a\oplus b))\qquad  a \in A, b\in B\,.
\]

(ii)  Although \BApSAy\ and \BApSCy\ are distinct notions (Proposition \ref{difference}), the same proof holds in both cases.
Suppose that $X$ is an $(A\oplus B)$-Banach bimodule, and $D:A\oplus B\to X$ is a continuous derivation.  Clearly $X$ is naturally both an $A$- and a $B$- Banach bimodule, and $D|_{A\oplus 0}\to X$, $D|_{0 \oplus B}\to X$ are continuous derivations.  Thus there are nets $(x_{i}), (y_{i}), (u_{j}), (v_{j}) \subset X$ and a constant $K > 0$ such that
\begin{eqnarray*}
D(a\oplus 0) &=&\lim_{i}(a\cdot x_{i} - y_{i}\cdot a)\quad(a\in A,,\\
D(0\oplus b) &=&\lim_{j}(b\cdot u_{j} - v_{j}\cdot b)\quad(b\in B)\,,\\
&&\kern -1.5cm\|a\cdot x_{i} - y_{i}\cdot a\|\leq K\|a\|\,,\quad (a\in A)\,,\\
&&\kern -1.5cm\|b\cdot u_{j} - v_{j}\cdot b\|\leq K\|b\|\,,\quad(b\in B)\,.
\end{eqnarray*}

Let $(e_{\alpha})$ (resp. $(f_{\beta})$) be multiplier bounded central approximate identities in $A$ (resp. $B$), with multiplier bound $M$.  Take $\varepsilon > 0$ and $a_{1}, \ldots a_{k}\in A$, $b_{1}, \ldots b_{k}\in B$. Then as in (i), there are $\alpha,  \beta, \gamma, \delta$ and $x, y, w, z\in X$ such that for $i=1, \ldots, k$, equations \eqref{firstsummand} and \eqref{secondsummand} hold,
together with, using centrality, 
\[
\|(a_{i}\oplus b_{i})(e_{\alpha}\cdot x) - ( y\cdot e_{\alpha})(a_{i}\oplus b_{i})\| = \|(a_{i}e_{\alpha})\cdot x- y\cdot(a_{i}e_{\alpha})\|\leq KM\|a_{i}\|\,,
\]
and
\[
\|(a_{i}\oplus b_{i})(f_{\beta}\cdot w) - ( z\cdot f_{\beta})(a_{i}\oplus b_{i})\| = \|(b_{i}f_{\beta})\cdot w - z\cdot(b_{i}f_{\beta})\|\leq KM\|b_{i}\|\,.
\]
Thus
\[
\|(a_{i}\oplus b_{i})(e_{\alpha} \cdot x+ f_{\beta} \cdot w)- ( y\cdot e_{\alpha} + z\cdot f_{\beta})(a_{i}\oplus b_{i})\|\leq KM\|a_{i}\oplus b_{i}\|\,.
\]

The result follows as in (i).
\end{proof}

The converse holds by Theorem \ref{quotients}.

 \begin{corollary}
 The class of \ApsA\ Banach algebras is closed under finite direct sums.  In particular, the direct sum of two \ApA\ algebras is \ApsA, though it may fail to be \ApA. \qed
 \end{corollary}
 
The arguments of the next two results are very similar, but there are subtle differences which raise further questions.  For a  set $S$, denote by $\cF$ the collection of the finite subsets of $S$, directed by set inclusion.  For $F\in \cF$, set $P_{F}$ to be the characteristic function of $F$.    Thus $(P_F)$ is  the standard  multiplier-bounded approximate identity of $\ell^1(S, \bC)$. 
 Note that the analogue of Theorem \ref{sums} for \ApA\ algebras is true in the unital case (proof follows), but is false in general (even for finite sums \cite[Theorem 4.1]{GhaR1}).
 
 \begin{theorem}\label{appsums}
Suppose that $(A_{\lambda})_{\lambda\in \Lambda}$ is a family of unital \ApA\ algebras.  Then  $c_{0}(A_{\lambda})$ is  approximately amenable. $\ell^{p}(A_{\lambda})$ is not \ApA\ if $\Lambda$ is infinite,  but is always approximately semi-amenable.
\end{theorem}

\begin{proof}  This is a variant of \cite[Example 6.1]{GhaLoy} using \cite[Proposition 6.1]{GhaLoyZh}.  By the latter, finite sums of the $A_{\lambda}$'s are \ApA.  Write $A = c_{0}(A_{\lambda})$.  Let $D:A \to X^{*}$ be a derivation.  For a finite set $F\subset \Lambda$, define $E_{F}$ by 
\[
E^{F}_{\lambda} = 
\begin{cases} 
e_{\lambda} & \lambda \in F\\
0 & \lambda \not\in F\,,\\
\end{cases}
\]
where $e_{\lambda}$ denotes the identity of $A_{\lambda}$.  Set $B^{F} =  E^{F}A$,  and define
$D^{F} (a) = D(E^{F}(a))$. Then by \cite[Lemma 2.4]{GhaLoy}, 
$D^{F} = {\rm ad}_{\eta^{F}} + {\rm st}-\lim_{i} {\rm ad} (\xi^{F}_{i})$ where $\eta^{F}$ is bounded over $F$.  Here $\xi_{i}^{F}, \eta^{F} \in E^{F}X^{*}E^{F}$, and
\[
\|{\rm ad}_{\eta^{F}} (E^{F}a - a)\|\to 0 \qquad (a\in A)
\]
by boundedness of $(\eta^{F})$, and $(E^{F})$ being a \BAI\ for $A$.
Thus
\begin{eqnarray*}
D(a) &=& \lim_{F}(D^{F}(a)) = \lim_{F}(\lim_{i}(aE^{F}\cdot \xi_{i}^{F} - \xi_{i}^{F}\cdot E^{F}a) + {\rm ad}_{\eta^{F}}a))\\
&=&  \lim_{F}\lim_{i}\Big(a\cdot \xi_{i}^{F} - \xi_{i}^{F}\cdot a +a\cdot\eta^{F} - \eta^{F}\cdot a)\Big)\,.\\
\end{eqnarray*}

Thus given $\varepsilon > 0$, and $a_{1}, \ldots, a_{k}\in A$, there is $\xi_{i}^{F} +\eta_{F}\in X^{*}$ such that
\[
\|D(a) -\Big(a_{i}\cdot (\xi_{i}^{F} + \eta^{F}) - (\xi_{i}^{F}+\eta^{F})\cdot a_{i}\Big)\| < \varepsilon \qquad i=1, \ldots, k\,.
\]
For infinite $\Lambda$,  $\ell^{p}(A_{\lambda})$ contains an isometric copy of $\ell^{p}$, and so has a SUM configuration, whence it fails to be \ApA, \cite[Theorem 2.5]{CG}.


The final statement follows from Theorem \ref{sums}.

\end{proof}


\begin{theorem}\label{sums} 
Suppose that $(A_{\lambda})_{\lambda\in \Lambda}$ is a family of \ApsC\ (semi-amenable) algebras.  Then $\ell^{p}( A_{\lambda})$ and $c_{0}(A_{\lambda})$ are  \ApsC (semi-amenable).
\end{theorem} \begin{proof}
 

For $\lambda \in \Lambda$, let $(e_{\alpha}^{\lambda})$ and $(f_{\beta}^{\lambda})$ be left and right approximate identities for $A_{\lambda}$.

 Take  $F = \{\lambda_{1}, \ldots, \lambda_{|F|}\}\in \cF$. For $i=1, \ldots, |F|$, let $\Lambda_{i}$ and $\Lambda_{i}'$ be the directed sets for $(e_{\alpha}^{\lambda_{i}})$ and $(f_{\beta}^{\lambda_{i}})$, and set $\Lambda_{F} = \Lambda_{1}\times\cdots\times \Lambda_{|F|}$ and $\Lambda_{F}' = \Lambda_{1}'\times\cdots\times \Lambda_{|F|}'$ with the product directions.

For $\underline{\alpha} = \{\alpha_{1}, \ldots, \alpha_{|F|}\}\in \Lambda_{F}$, set
\[
e_{\underline{\alpha}}(\lambda)  =
\begin{cases}
e_{\alpha_{i}}^{\lambda_{i}} & \lambda\ = \lambda_{i}\\
0 & {\rm otherwise}.\\
\end{cases} 
 \]
 Similarly for $f_{\underline{\beta}}$.  Then for $a\in c_{0}(A_{\lambda})$, 
 \[
e_{\underline{\alpha}}(\lambda) a =   
\begin{cases}
e_{\alpha_{i}}^{\lambda_{i}}a_{\lambda_{i}} & \lambda = \lambda_{i}\\
0 & \rm otherwise.
\end{cases}
\]
 Thus
 \[
\|e_{\underline{\alpha}}(\lambda) a - a\| = \max_{1\leq i\leq |F|}| \|e_{\alpha_{i}}^{\lambda_{i}}a_{\lambda_{i}}-a_{\lambda_{i}}\| + \sup_{\lambda\not\in F} \|a_{\lambda}\|\,.
 \]
Given $\varepsilon > 0$ first choose $F$ such that the second term is less than $\varepsilon/2$, then, for this fixed $F$, take $\alpha_{i}^{\lambda_{1}}, \ldots, \alpha_{|F|}^{\lambda_{|F|}}$ such that the first term is less than $\varepsilon /2$.  Then $\|e_{\underline{\alpha}}(\lambda) a - a\| <\varepsilon$, and hence $(e_{\underline{\alpha}})$ is a left approximate identity for $c_{0}(A_{\lambda})$.  Similarly for $(f_{\underline{\beta}})$ on the right.

For a continuous derivation $D: c_{0}(A_{\lambda}) \to X$
into a Banach $c_{0}(A_{\lambda})$-bimodule, its restriction to  $P_{F}c_{0}(A_{\lambda})$ is approximately semi-inner, implemented by $(\xi_{i}^{F}), (\eta_{i}^{F})$.
Thus for $a\in c_{0}(A_{\lambda})$,    

\begin{eqnarray*}
D(a) &=& \lim_F D(P_F a) =\lim_{F} \lim_{i}(P_{F}a\cdot \xi_{i}^{F} - \eta_{i}^{F}\cdot P_{F}a)\\
&=& \lim_F\lim_{i}\lim_{\underline{\alpha}, \underline{\beta}} \Big(P_F af_{\underline{\beta}}\cdot\xi_i^{F} - \eta_i^{F}\cdot P_Fe_{\underline{\alpha}}a\Big)\\
&=& \lim_F \lim_{i}\lim_{\underline{\alpha}, \underline{\beta}} \Big(a\cdot(P_F f_{\underline{\beta}}\cdot \xi_i^{F}) - (\eta_{i}^{F}\cdot P_Fe_{\underline{\alpha}})\cdot a\Big)\,.
\end{eqnarray*}

That $D$ is approximately semi-inner  follows as in previous arguments.
\end{proof}

\begin{remark}  The use of approximate identities is only to enable 
the shifting of the action of $P_{F}$ away from $a$.

\end{remark}
A similar argument shows the following. 


\begin{theorem}\label{simpletensor} Suppose that $A$ is an  \ApsA\ Banach algebra with a central multiplier-bounded approximate identity.  Let $B$ be a Banach function algebra which has a multiplier-bounded approximate identity of elements of finite support.  Then $A \tensor B$ is approximately semi-amenable.
\end{theorem}
 
\begin{proof}
Let $K$ be a multiplier-bound for the central approximate identity $(e_{\alpha})$ of $A$ and $(f_{\beta})$ of $B$.  Set $(\alpha, \beta)\leq (\alpha', \beta')$ to mean $\alpha\leq \alpha'$ and $\beta \leq \beta'$. Let   $X$ be $A\tensor B$--bimodule, $D:A\tensor B\to X^{*}$ a  continuous derivation.

Then, given $\sum_{k=1}^{\infty}a_{k}\otimes b_{k} \in A\tensor B$ we have
\begin{eqnarray*}
\|\sum_{k=1}^\infty  e_{\alpha}a_k \otimes f_{\beta} b_k - \sum_{k=1}^{\infty} a_k \otimes b_k\|
&\leq&
\|\sum_{k=1}^N   e_{\alpha}a_k \otimes f_{\beta} b_k - \sum_{k=1}^{N} a_k \otimes b_k\| \\
&&+  (K ^{2}+1)\sum_{k=N+1}^\infty \|a_k\| \|b_k\|\,.
\end{eqnarray*}
So given $\varepsilon > 0$ we can take $N$ such that the second sum is less than $\varepsilon/2$. Then for this $N$ we can choose $\alpha$ and $\beta$ such that the first (finite) sum is less than $\varepsilon/2$.  Thus

\begin{equation}\label{tensorcgce}
\sum_{k=1}^\infty  e_{\alpha}a_k \otimes f_{\beta} b_k \to \sum_{k=1}^{\infty} a_k \otimes b_k\,.
\end{equation}

For each $\beta$, define $n_{\beta} = |{\rm supp}(f_{\beta})|$. Then $f_{\beta}B\simeq \mathbb C^{n_{\beta}}$ so that $A\tensor f_\beta B\simeq A^{n_{\beta}}$  and so is \ApsA.  Thus there are $(\xi_{i}), (\eta_{i})\in X$ such that, for each $a =\sum_{k} a_{k}\otimes  b_{k}\in A\tensor B$, 
\begin{eqnarray*}
D(\sum_k a_k\otimes  f_{\beta}b_k) = \lim_i \left((\sum_k a_k\otimes f_{\beta}b_k)\cdot \xi_i -\eta_i\cdot ( \sum_k a_k\otimes f_{\beta}  b_k)\right)\,.\end{eqnarray*}
So by \eqref{tensorcgce}
\begin{eqnarray*}
D(\sum_k a_k\otimes b_k)
&=& \lim_{\alpha, \beta} \lim_i \left((\sum_k e_{\alpha}a_k \otimes f_{\beta} b_k)\cdot \xi_i -\eta_i\cdot ( \sum_k e_{\alpha}a_k \otimes f_{\beta} b_k)\right)\\
&& \kern -1.75cm = \lim_{\alpha, \beta} \lim_i \left(\sum_k a_k\otimes b_k\cdot(e_{\alpha} \otimes f_{\beta}\cdot \xi_i) -(\eta_i\cdot e_{\alpha}\otimes f_{\beta} \cdot  \sum_k a_k\otimes b_k\right)
\end{eqnarray*}

The result follows as before.
\end{proof}

\begin{corollary}
For each $n\geq 2$, the algebra $A(\mathbb F_{n})\tensor A(\mathbb F_{n})$ is \ApsA.
\end{corollary}  
\begin{proof}
By Corollary \ref{F2}, $A(\mathbb F_{n})$ is \ApsA, and as noted there, it has a multiplier bounded approximate identity of elements in $c_{00}$.  Thus Theorem \ref{simpletensor}  applies.
\end{proof}

\begin{remark}
The argument of Theorem \ref{simpletensor} works for $B =c_0(\Lambda)$ and $B = \ell^{p}(\Lambda)$.  In particular, $\ell^{p_{1}}\tensor \ell^{p_{2}}$ is \ApsA\ for $1\leq p_{1}, p_{2} <\infty$.  Are these algebras \BApsA? (Yes if $p_{1} = p_{2}=1$.)

Unfortunately this method sheds no light on the situation for general \ApsA\ $B$.  

Note that the example in \cite[Theorem 2.3]{GhaLoy2}  shows that the tensor product of \ApsA\ algebras is not  \ApsA\ in general, since  for algebras with identity, \ApAy\ and \ApsA\ are the same, Proposition \ref{equality1}.

 \end{remark}
 The notion of semi-inner derivations arose in consideration of \ApAy\ in the context of tensor products, \cite[\S4]{GhaLoy2}.  In fact the following follows by obvious minor adjustments to the proof \cite[Theorem 4.1]{GhaLoy2}.
 
  \begin{theorem}\label{tensor}
 Suppose that
  Let $A$ and $B$ be non-zero Banach algebras.  Suppose that $A\tensor B$ is 
  \begin{enumerate}
  \item[(i)] approximately semi-amenable, or
  \item[(ii)] boundedly approximately semi-amenable, or
  \item[(iii)] boundedly approximately semi-contractible.
  \end{enumerate}
   
  Then $A$ and $B$ have the same property. \hfill $\Box$
  \end{theorem}
   



If the  boundedness of the \AI\ is dropped in Theorem \ref{parallel}, then in certain cases there is a weaker result. 

\begin{theorem}\label{cai} Suppose that $A$ is \ApsA\,, $J$ is a closed (two-sided) ideal in $A$ with an approximate identity consisting of central idempotents.  Then $J$ is approximately semi-amenable.
\end{theorem}
 
 \begin{proof}
Let $(e_{\alpha})$ be the approximate identity for $J$ consisting of central idempotents.  
 Given $\alpha$, $e_{\alpha}J = e_{\alpha}Je_{\alpha}$ is a closed two-sided ideal with identity $e_{\alpha}$.  By Theorem \ref{parallel}, $e_{\alpha}J$ is \ApsA.  Let $X$ be a Banach $J$-bimodule, $D:J\to X$ a continuous derivation.  Then $D|_{e_{\alpha}J}$ is approximately semi-inner, so there is a nets $(\xi_{\alpha, i})_{i}, (\eta_{\alpha, i})_{i}\subset X$ such that
\[
D(e_{\alpha}a) = \lim_{i}(e_{\alpha}a\cdot \xi_{\alpha, i})- \eta_{\alpha, i}\cdot e_{\alpha}a)\,, \quad (a\in J)\,.
\]
Thus
\begin{eqnarray*}
D(a) &=& \lim_{\alpha}\lim_{i}(e_{\alpha}a\cdot \xi_{\alpha, i})- \eta_{\alpha, i}\cdot e_{\alpha}a)\\
&=&\lim_{\alpha} \lim_{i}\Big(a\cdot (e_{\alpha}\cdot \xi_{\alpha, i})- (\eta_{\alpha, i}\cdot e_{\alpha})\cdot a\Big)
\end{eqnarray*}
\end{proof}

The following is a special case of Theorem \ref{directsums}, but indicates another trick for avoiding difficulties with \ApA\ algebras.

\begin{corollary} Suppose that $A$ and $B$ are \ApA, with each having an  approximate identity consisting of central idempotents.  Then $A\oplus B$ is approximately semi-amenable.
\end{corollary}
 \begin{proof}
 We have $A^{\#}$ and $B^{\#}$ are \ApA\,, whence so is $A^{\#} \oplus B^{\#}$.  But $A\oplus B$ is a closed two-sided ideal in $A^{\#} \oplus B^{\#}$.  Applying Theorem \ref{cai} gives the result.
 \end{proof}

\begin{example} \label{Feinstein}
\[
\mathcal{A} = \{ x =(x_{n})\in c_{0} : \|x\| := \|x\|_{\infty} + \sum_{n=1}^{\infty} |x_{n+1} - x_{n}| <\infty\}\,,
\]
 which is \ApA,  \emph{cf.} Corollary \ref{maxmin}.  In \cite[\S4]{GhaLoyZh} it is noted that for any sequence $(m_{j})\subset \bN$ with $m_{k} > m_{k-1} +1$, the ideal
 \[
I =\big\{x\in \mathcal{A} : x_{j} = 0\ \  {\rm unless}\  j\in \{m_{k}\} \big\}
\]
 is a complemented ideal isomorphic to $\ell^{1}$.  Then $\mathcal{A}$ and $I$ satisfy the hypotheses of Theorem \ref{cai}, and so $\ell^{1}$ is approximately semi-amenable, as already noted in \cite[Example 3.7]{GhaLoy2}.
 
\end{example}

\begin{example}  For $1 < p < \infty$, the James algebra $J_{p}$ is the subalgebra of $c_{0}(\bN)$ consisting of all sequences $a = (a_{n})$ such that 
\[
 N(a):=\sup  \left [ \sum_{i=1}^{k-1} \left|a_{n_{i+1}} -a_{\eta_{i}}\right|^{p} + \left|a_{n_{k}} - a_{n_{1}}\right|^{p}\right] < \infty\,,
\]
where the supremum is taken over all finite sequences $n_{1} < n_{2} < \cdots  < n_{k}$ in $\bN$. The norm is $\|a\|_{p} = 2^{-1/p}N(a)^{1/p}$\,.

The ideals of $J_{p}$ are just the kernels of subsets of $\bN$, \cite[Lemma 1.1]{White}, and have a \BAI\ if and only if they are of finite or cofinite dimension  \cite[Proposition 2.2]{White}.     Clearly finite-dimensional ideals are in fact amenable, and they are the only ones \cite[Theorem 3.3]{White}.  The cofinite-dimensional  ones are isomorphic to $J_{p}$ and so 
by Theorem \ref{BFAG} are \ApA.  The question remains, however, for the other ideals (infinite dimension and codimension).
For these, \cite[Theorem 4.3]{White} shows they have $\ell^{p}$ as a quotient, and so fail to be \ApA\ \cite[Theorem 4.1]{DLZ}.

 On the other hand, if $I$ has kernel $K$, then by \cite[Lemma 2.5]{AAG},  $\left(\sum_{i=1}^{m}\delta_{i}\right)(1-\chi_{K})$ gives an (unbounded) approximate identity for $I$, with multiplier bound $1$.  In fact, $\|a- a\left(\sum_{i=1}^{m}\delta_{i}\right)(1-\chi_{K})\|\leq \|a\|$, so referring to Theorem \ref{BFAG},  $K_{BASC}(J_{p})\leq 2$.  Proposition \ref{lowerbound} shows that $2$ is the exact value.


\end{example}
 
\oomit{
 \section{Reformulation of criteria}   
 Note that we could have used the equivalent hypothesis of \ACy\ in the statement of this result -- it is used in the proof.
 \begin{theorem}\label{otherform}
The Banach algebra $A$ is approximately semi-amenable if and only if there exist nets $(\xi_{i}), (\eta_{i})$ in  $A\tensor A$, $(f_{i}), (g_{i}), (h_{i}), (k_{i})$ in $A$ such that for $a\in A$,
 \begin{enumerate}
 \item[(i')] $ a\cdot \xi_{i } - \eta_{i}\cdot a -a\otimes g_{i} +h_{i}\otimes a \to 0$\,,
  \item[(ii')]$af_{i}\to a$, $k_{i}a\to a$\,, 
\item[(iii')] $\pi(\xi_{i}) - f_{i} -g_{i} \to 0$\,,
  \item[(iv')]  $\pi(\eta_{i}) - h_{i} -k_{i} \to 0$\,.
  \end{enumerate}
\end{theorem}
 \begin{proof}
If $A$ is approximately semi-amenable, then Theorems \ref{equiv} and \ref{semicont} gives nets $(\xi'_{i}), (\eta'_{i})$ in $A^{\#}\tensor A^{\#}$ such that, with $e$ the adjoined identity,   
 \begin{enumerate}
\item[(i)]  $a\cdot \xi'_i - \eta'_i \cdot a \to 0\,,\quad (a\in A)$\,,
\item[(ii)] $ \pi(\xi'_i) \to e$ and $ \pi(\eta'_i) \to e$\,.
\end{enumerate}
Now, taking $\|\cdot\|_{1}$ between the summands,
 \[ 
 A^{\#}\tensor A^{\#} = (A\tensor A) \oplus (e\otimes A) \oplus (A\otimes e) \oplus \C (e\otimes e)\,,
  \]
 and this is isometric. So we can write
  \[
 \xi'_{i} = \xi_{i } - f_{i}\otimes e - e\otimes g_{i} + \alpha_{i}e\otimes e\,,
 \]
 \[
 \eta'_{i} = \
 \eta_{i }  - h_{i}\otimes e - e\otimes k_{i} + \beta_{i}e\otimes e\,.
\] 
 Thus we can rewrite (i) as
\[
 a\cdot\Big(\xi_{i } - f_{i}\otimes e - e\otimes g_{i} + \alpha_{i}e\otimes  e\Big)
  -\Big( \eta_{i }  - h_{i}\otimes e - e\otimes k_{i} + \beta_{i}e\otimes e\Big)\cdot a \to \,.
\]
 That is
 \begin{equation}\label{comp}
 ( a\cdot \xi_{i } - \eta_{i}\cdot a) + (-a\otimes g_{i} +h_{i}\otimes a) + (-af_{i}\ +\alpha_{i}a)\otimes e + e\otimes(-\beta_{i}a +k_{i}a) \to 0\,.
\end{equation}
 Further, (ii) gives
\[
\pi(\xi_{i}) - f_{i} -g_{i} +\alpha_{i}e\to e, \ \  \pi(\eta_{i}) - h_{i} -k_{i} +\beta_{i}e\to e\,,
\]
whence $\alpha_{i}\to 1$, $\beta_{i}\to 1$, and so 
 \begin{enumerate}
 \item[(i')] $ a\cdot \xi_{i } - \eta_{i}\cdot a -a\otimes g_{i} +h_{i}\otimes a \to 0$\,,
  \item[(ii')]$af_{i}\to a$, $k_{i}a\to a$\,, 
\item[(iii')] $\pi(\xi_{i}) - f_{i} -g_{i} \to 0$\,,
  \item[(iv')]  $\pi(\eta_{i}) - h_{i} -k_{i} \to 0$\,.
  \end{enumerate}
 Conversely, if (i'), $\ldots,$ (iv') hold then (i) and (ii) follow. 
 \end{proof}
\begin{remark}
A small perturbation to $f_{i}$ and $k_{i}$ shows the limits in (iii') and (iv') can be replaced by equalities without upsetting (ii').
\end{remark} 
 %
 In the case that $A$ is commutative, using the fact that $\iota:a\otimes b\mapsto  b\otimes a$ is an isometric isomorphism of $A\tensor A$, (i') gives
 \[
  \iota(\xi_i)\cdot a - a\cdot \iota(\eta_{i}) -g_{i}\otimes a +a \otimes h_{i}  \to 0\,. 
    \]
 Subtracting this from (i') yields
 \[
 a\cdot (\xi_{i} +\iota(\eta_{i})) - (\eta_{i} + \iota(\xi_{i})) \cdot a - a\otimes(g_{i} + h_{i}) +(h_{i} + g_{i})\otimes a \to 0\,.
 \]
 Setting $\xi_{i}' = (\xi_{i} +\iota(\eta_{i}))/2$, $\eta_{i}' = (\eta_{i} + \iota(\xi_{i}))/2, g_{i}' = (g_{i} + h_{i})/2, f_{i}' = (f_{i} + k_{i})/2$, we thus have
 \begin{enumerate}
 \item[(c1)]  $ a\cdot m_{i }' - \eta_{i}'\cdot a -a\otimes g_{i}' +g_{i}'\otimes a \to 0$\  $(a\in A)$\,,
 \item[(c2)]  $ af_{i}' \to a$\  $(a\in A)$\,,
 \item [(c3)]$ \pi(\xi_{i}') - f_{i}'-g_{i}' \to 0$\,,
  \item[(c4)] $\pi(\xi_{i}' )= \pi(\eta_{i}')$\,.
\end{enumerate}
 %
\vskip 1cm
 This time  \BACy\ cannot be replaced by \BAAy.
  \begin{theorem}\label{bddappsemicont}
The Banach algebra $A$ is boundedly approximately semi-contractible if and only if there exists $K > 0$ and nets $(\xi_{i}), (\eta_{i})$ in  $A\tensor A$, $(f_{i}), (g_{i}), (h_{i}), (k_{i})$ in $A$ such that for $a\in A$,
 \begin{enumerate}
 \item[(i')] $ a\cdot \xi_{i } - \eta_{i}\cdot a -a\otimes g_{i} +h_{i}\otimes a \to 0$\,,with
  $\|a\cdot m_{i } - \eta_{i}\cdot a -a\otimes g_{i} +h_{i}\otimes a\|\leq K\|a\|$\,,
  \item[(ii')]$af_{i}\to a$, $k_{i}a\to a$\,, with $\|af_{i}\|, \|k_{i}a\|\leq K\|a\|$\,,
\item[(iii')] $\pi(\xi_{i}) = f_{i} + g_{i}$\,,
  \item[(iv')]  $\pi(\eta_{i}) = h_{i} +k_{i}$\,.
  \end{enumerate}
\end{theorem}
 \begin{proof}
If $A$ is boundedly approximately semi-contractible, then Theorem \ref{bddsemicont} gives a constant $K > 0$ and nets $(\xi'_{i}), (\eta'_{i})$ in $A^{\#}\tensor A^{\#}$ such that, with $e$ the adjoined identity,
 \begin{enumerate}
 \item[(i)] $a\cdot \xi'_{i} - \eta'_i \cdot a \to 0\,, \|a\cdot \xi'_{i} - \eta'_{i}\cdot a\| \leq K\|a\| \quad (a\in A)$\,,
\item[(ii)] $ \pi(\xi'_i) = e$ and $ \pi(\eta'_i) = e$\,.
\end{enumerate}
 Arguing as in Theorem \ref{otherform}, we have nets $(\xi_{i}), (\eta_{i})$ in  $A\tensor A$, $(f_{i}), (g_{i}), (h_{i}), (k_{i})$ in $A$ with
 \marginpar{\bf \color{red}{small corrections herer}}
 \oomit{Now, taking $\|\cdot\|_{1}$ between the summands,
 \[ 
 A^{\#}\tensor A^{\#} = (A\tensor A) \oplus (e\otimes A) \oplus (A\otimes e) \oplus \C (e\otimes e)\,,
  \]
 and this is isometric.  So we can write
  \[
 \xi_{i} = m_{i } - f_{i}\otimes e - e\otimes g_{i} + \alpha_{i}e\otimes e\,,
 \]
 \[
 \eta_{i} = n_{i }  - h_{i}\otimes e - e\otimes k_{i} + \beta_{i}e\otimes e\,,
 \] 
 and so we can rewrite (i) as
\[
 a\cdot\Big(m_{i } - f_{i}\otimes e - e\otimes g_{i} + \alpha_{i}e\otimes  e\Big)
  -\Big( n_{i }  - h_{i}\otimes e - e\otimes k_{i} + \beta_{i}e\otimes e\Big)\cdot a \to 0
\]
 That is
 \begin{equation}\label{compbdd}
 ( a\cdot m_{i } - \eta_{i}\cdot a) + (-a\otimes g_{i} +h_{i}\otimes a) + (-af_{i}\ +\alpha_{i}a)\otimes e + e\otimes(-\beta_{i}a +k_{i}a) \to 0\,.
\end{equation}
 Further, (ii) gives
\[
\pi(\xi_{i}) = f_{i} +g_{i} , \ \  \pi(\eta_{i}) =h_{i} +k_{i}\,.
\]
Thus
}
\begin{enumerate}
 \item[(i')] $ a\cdot \xi_{i } - \eta_{i}\cdot a -a\otimes g_{i} +h_{i}\otimes a \to 0$\,,
 $\|a\cdot \xi_{i } - \eta_{i}\cdot a -a\otimes g_{i} +h_{i}\otimes a\|\leq K\|a\|$\,,
  \item[(ii')]$af_{i}\to a$, $k_{i}a\to a$\,, $\|af_{i}\|, \|k_{i}a\| \leq (K+1)|a\|$\,,
\item[(iii')] $\pi(\xi_{i}) = f_{i} +g_{i}$\,,
  \item[(iv')]  $\pi(\eta_{i}) = h_{i} +k_{i}$\,.
  \end{enumerate}
 Conversely, if (i'), $\ldots,$ (iv') hold then (i) and (ii) follow with a new $K$\,. 
 \end{proof}
}
 
\section{The uniform case} 

We could  make the following definition.
 \begin{definition}\label{unifsemiamen}
The Banach algebra $A$ is \emph{uniformly approximately semi-amenable} if for any Banach $A$-bimodule $X$, any continuous derivation $D :A\to X^{*}$ is norm approximable by semi-inner mappings.
\end{definition}

But there is no need:

\begin{theorem}\label{unifsamen}
Suppose that the Banach algebra $A$ is uniformly approximately semi-amenable.  Then $A$ is amenable.
\end{theorem}

\begin{proof}
Consider the module $A^{**}$ with natural right action and trivial left action, so the inclusion $j:A\hookrightarrow A^{**}$ is a derivation.  Thus there is $t\in A^{**}$ such that  $\|at - a\|  \leq \|a\|/2$, $a\in A$.  Taking the first Arens product on $A^{**}$, 
weak$^{*}$-continuity shows  the inequality remains valid for $a\in A^{**}$.

So right multiplication by $t$ on $A^{**}$, $\rho_{t}$, satisfies $\|\rho_{t} - {\rm id}\| \leq 1/2$, and so $\rho_{t}$  is invertible, whence there is $b\in A^{**}$ with $bt =t$.  But then for any $a\in A^{**}$, $(ab- a)t = 0$, whence $ab = a$.  Thus $A^{**}$ has a right identity.  But then $A$ has a right \BAI.
 Now apply the same argument to $A^{\rm{op}}$ so that $(A^{\rm op})^{**}$ has a right identity, and so $A^{\rm{op}}$  has a bounded right approximate identity, that is, $A$ has a bounded left approximate identity $(e_{\alpha})$.

 Suppose that $D:A\to X^{*}$ is a continuous derivation for some  Banach $A$-bimodule $X$.  By a standard reduction, we may suppose that $X$ is neo-unital. By hypothesis on $A$  there are sequences $(\xi_{n}), (\eta_{n})$ in $X^{*}$ such that
\[
a\cdot \xi_{n} - \eta_{n}\cdot a \to D(a) \ {\rm uniformly\ on\ } \|a\|\leq 1\,.
\]
Similar to equation \eqref{limits} following Definition \ref{asi}, it follows that
\begin{equation}\label{uniformcgce}
a\cdot (\xi_{n} - \eta_{n})\cdot b \to 0  \ {\rm uniformly\ on\ } \|a\|\leq 1, \|b\|\leq 1\,.
\end{equation}

Given $\varepsilon > 0$, take $n_{0}$ such that $\|a\cdot (\xi_{n} - \eta_{n})\cdot a\| < \varepsilon \|a\|$ for $a\in A$,   $n > n_{0}$.
Then, for $x\in X$, 
\begin{eqnarray*}
\langle \xi_{n} - \eta_{n} , x \rangle &=& \lim_{\alpha} \langle \xi_{n} - \eta_{n} ,e_{\alpha} x e_{\alpha}\rangle\\
&=& \lim_{\alpha} \langle e_{\alpha} (\xi_{n} - \eta_{n})e_{\alpha}, x\rangle\\
 &\leq& \limsup_{\alpha} \|e_{\alpha} (\xi_{n} - \eta_{n})e_{\alpha}\| \|x\|\,.
\end{eqnarray*}

From equation (\ref{uniformcgce}) the first factor converges to zero as $n\to \infty$.  It follows that $\|\xi_{n} - \eta_{n}\|\to 0$.

So we are in fact in the uniformly \ApA\ case, whence the result by \cite[Theorem 3.2]{GhaLoyZh}.\end{proof}

\begin{theorem}\label{unifscont}
Suppose that the Banach algebra $A$ is uniformly approximately semi-contractible.  Then $A$ is contractible.
\end{theorem}
\begin{proof}  A simpler version of the above shows  that $A$ is uniformly approximately contractible.  Now use \cite[Theorem 4.1]{GhaLoy}.
\end{proof}

\oomit{
Alternative proof of Theorem \ref{unifsamen}, using the same ideas.

\begin{theorem}\label{uniform}
 Suppose that a Banach algebra $A$ is uniformly approximately semi-amenable.  Then $A$ is amenable.
 \end{theorem}
 \begin{proof}   The argument is close to that of \cite[Theorem 3.2]{GhaLoyZh}, but there are subtle changes needed as we cannot reduce to the unital case, cf. Proposition \ref{equality1}.
 Let $\pi : A^{\#}\tensor A^{\#}$ be the product map, $K_{0} ={\rm ker}\pi$ the diagonal ideal.  By \cite[Theorem VII.2.20]{HEL} and \cite[Proposition 11.4]{BandD}  we need to show that $K_{0}^{**}$ has a right identity.
Let $S =\displaystyle \sum_{j} a_{j}\otimes b_{j}\in K_{0}$, $t, u\in K_{0}^{**}$.  Then $ \displaystyle \sum_{j} a_{j} b_{j} = 0$, so by adding in $0$ judiciously, 
\begin{eqnarray*}
st-s &= &(\sum_{j} a_{j}\otimes b_{j})t - u\sum a_{j} b_{j} - \sum_{j} a_{j}\otimes b_{j} + e\sum_{j} a_{j}b_{j}\\
&= & \sum_{j} \Big(a_{j}\cdot t -u\cdot a_{j} -a_{j}\otimes e + (e \otimes a_{j})\cdot b_{j}\Big)\,.
\end{eqnarray*}
So that
\[
\|st - s\| \leq \left(\sum_{j} \|a_{j}\| \|b_{j}\|\right) \sup_{ \|a\| = 1} \| a\cdot t - u\cdot a - a\otimes e + e\otimes a\|\,,
\]
and hence
\begin{equation}\label{tensor bound}
\|st - s\| \leq \|s\| \sup_{ \|a\| = 1} \| a\cdot t - u\cdot a - a\otimes e + e\otimes a\|\,.
\end{equation}
Weak$^{*}$ continuity shows  \eqref{tensor bound} remains valid for $s\in K_{0}^{**}$.
Now take the derivation $D :A\to K_{0}^{**}$ give by $D(a) = a\otimes e - e\otimes e$.  By hypothesis, there are  sequences $(t_{n)}, (u_{n})\in K_{0}^{**}$ such that for all $a\in A$,
\[
\|a\cdot t_{n} - u_{n}\cdot a -a \otimes e + e\otimes a\| \leq \|a\|/n\,.
\]
 In particular, there are $t, u\in K_{0}^{**}$ with   \marginpar{\bf $u$ turns out to be irrelevant!}
\[
 \|a\cdot t - u\cdot a -a \otimes e + e\otimes a\| \leq \|a\|/2\,.
 \]
From \eqref{tensor bound} right multiplication by $t$ on $K_{0}^{**}$, $\rho_{t}$, satisfies $\|\rho_{t} - {\rm id}\| < 1$, and so $\rho_{t}$  is invertible, whence there is $x\in K_{0}^{**}$ with $xt =t$.  But then for any $y\in K_{0}^{**}$, $(yx - y)t = 0$, whence $yx = y$.  Thus $K_{0}^{**}$ has a right identity, as required.
 \end{proof}
Reduction proof:  Looking at the module $A^{**}$ with natural right action and trivial left action, gives the same inequality as above: there is $t\in A^{**}$  such that
\[
 \|st - s\|\leq \|s\|/2\,, \quad s\in A\,.
 \] 
Then the invertibility argument gives $A^{**}$ has a right identity, whence $A$ has a bounded right approximate identity.  Now apply the same argument to $A^{\rm{op}}$ so that $(A^{\rm op})^{**}$ has a right identity, and so $A^{\rm{op}}$  has a bounded right approximate identity, that is $A$ has a bounded left approximate identity.  Now by Proposition 2.2  we have the approximate amenable case, \cite[Theorem 3.2]{GhaLoyZh}.
 }

\begin{corollary}\label{findim}
 Suppose that $A$ is a  finite-dimensional Banach algebra and is  approximately semi-amenable. Then $A$ is amenable.
 \end{corollary}
 
 \begin{proof} The argument for the uniform case shows only a single derivation into a finite-dimensional module needs to be considered.  But in the finite-dimensional case  $A^{**} = A$, and the strong-operator topology coincides with the norm topology.
 \end{proof}

 \section{Some open problems}
 \begin{enumerate}
 \item  Are the Schatten classes $S_{p}$ \ApsA?\\
 
 \item  What is the connection between \ApsAy and pseudo-amenability?\\



 \item Suppose that $A$ is \ApsA, and has a bounded approximate identity.  What can be said further?  Proposition \ref{cbaistrong} gives a partial answer. \\

  \item  What is the situation regarding approximate semi-amenability for $C^{*}$-algebras?  Note that $B(H)$ fails to be \ApA\ by the argument of \cite{Ozawa}, and the Choi algebra fails to be \ApA\ by the argument of  \cite{MDChoi}. 
  So both fail to be \ApsA\ by Theorem \ref{semi+baa=approx}.\\

 \item  Is a complemented ideal in an  \ApsA \  Banach algebra itself  \ApsA ? \\
 
 \item  We have seen in Proposition \ref{difference} that \BApSCy\ and \BApSAy\ are distinct notions, and in Theorem \ref{equiv} that \ApsCy\ and \ApsAy\ are the same.   Are \BApSAy\ and \ApsAy\ distinct?
 
\end{enumerate}

\end{document}